\documentclass[a4paper,11pt]{article}
\usepackage{graphicx}
\usepackage{amsmath,amsthm,amssymb,enumerate}
\usepackage{euscript,mathrsfs}
\usepackage{dsfont}
\usepackage{xcolor}
\usepackage{empheq}
 \usepackage[left=2cm,right=2cm,top=3.5cm,bottom=3.5cm]{geometry}
\usepackage[linkcolor=blue,colorlinks=true]{hyperref}

\usepackage{tcolorbox}

\usepackage{stmaryrd}
\allowdisplaybreaks

\usepackage{color}
\catcode`\@=11 \@addtoreset{equation}{section}

\catcode`\@=12

\usepackage[labelformat=simple]{subcaption}

\usepackage{enumitem}
\usepackage{url}

\newtheorem{Theorem}{Theorem}[section]

\newtheorem{Lemma}[Theorem]{Lemma}

\theoremstyle{definition}
\newtheorem{Definition}[Theorem]{Definition}

\newtheorem{Remark}[Theorem]{Remark}

\newcommand{\bTheorem}[1]{
\begin{Theorem} \label{T#1} }
\newcommand{\eT}{\end{Theorem}}

\newcommand{\cch}{\cc_h}

 \newcommand{\cc}{\mathfrak{c}}
 \newcommand{\CC}{\mathfrak{C}}
\newcommand{\vrh}{\vr_h}
\newcommand{\vuh}{\vu_h}
\newcommand{\auh}{\widehat{\vuh}}
\newcommand{\auhk}{\widehat{\vuh^k}}
\newcommand{\Piw}{\Pi^W_h}

\newcommand{\Uh}{U_h}

\newcommand{\Del}{\Delta_x}

\newcommand{\bu}{\mathbf u}

\newcommand{\bfe}{\mathbf{e}}

\newcommand{\bfphi}{\boldsymbol{\phi}}

\newcommand{\Piq}{\Pi_h^Q}
\newcommand{\Piv}{\Pi_h^V}
\newcommand{\Pix}{\Pi_h^X}

\newcommand{\Fc}{ f  }

\newcommand{\up}{{\rm up}}
\newcommand{\Up}{{\rm Up}}
\newcommand{\bUp}{\textbf{Up}}
\newcommand{\Fup}{{\rm F}_\eps^{\up}}
\newcommand{\bFup}{\textbf{F}_\eps^{\up}}

\newcommand{\ds}{\,{\rm d}S_x}

\newcommand{\bFormula}[1]{
\begin{equation} \label{#1}}
\newcommand{\eF}{\end{equation}}

\newcommand{\grid}{\mathcal{T}}

\newcommand{\TS}{\Delta t}

\newcommand{\Divh}{{\rm div}_h}
\newcommand{\Gradh}{\nabla_h}

\newcommand{\Laph}{\Delta_h}
\newcommand{\Lap}{\Delta_x}

\newcommand{\co}[2]{{\rm co}\{ #1 , #2 \}}
\newcommand{\Ov}[1]{\overline{#1}}
\newcommand{\avs}[1]{ \left\{\hspace{-0.2em}\left\{#1 \right\} \hspace{-0.2em}\right\}    }

\newcommand{\DC}{C^\infty_c}

\newcommand{\aleq}{\stackrel{<}{\sim}}
\newcommand{\ageq}{\stackrel{>}{\sim}}

\newcommand{\vr}{\varrho}

\newcommand{\vv}{\vc{v}}
\newcommand{\vu}{\vc{u}}
\newcommand{\vm}{\vc{m}}
\newcommand{\vs}{\vc{v}_\sigma}
\newcommand{\us}{\vc{u}_\sigma}

\newcommand{\EE}{\mathfrak{E}}
\newcommand{\RR}{\mathfrak{R}}

\newcommand{\vn}{\vc{n}}

\newcommand{\vc}[1]{{\bf #1}}

\newcommand{\Div}{{\rm div}_x}

\newcommand{\Grad}{\nabla_x}

\newcommand{\dx}{\,{\rm d} {x}}

\newcommand{\dt}{{\rm d} t }

\newcommand{\jump}[1]{\left\llbracket#1\right\rrbracket}
\newcommand{\abs}[1]{\left\lvert#1 \right\rvert}

\newcommand{\norm}[1]{\lVert#1\rVert}
\newcommand{\normH}[1]{  \lvert \lVert#1\rVert \rvert }
\newcommand{\normB}[1]{\lvert\lVert#1\rVert \rvert _B}

\newcommand{\vU}{\vc{U}}

\newcommand{\dxdt}{\dx \dt}
\newcommand{\intO}[1]{\int_{\Omega} #1 \dx}
\newcommand{\intOB}[1]{\int_{\Omega} \left( #1  \right)\dx}

\newcommand{\intSh}[1] {\int_{\sigma} #1 \ds }

\newcommand{\intfaces}{ \sum_{ \sigma \in \faces } \intSh}

\newcommand{\intfacesB}[1]{ \sum_{ \sigma \in \faces } \intSh{\left( #1 \right)}}

\newcommand{\intK}[1]{\int_{K} #1 \ \dx}
\newcommand{\intTN}[1]{\int_{\Omega} #1 \ \dx}
\newcommand{\intTO}[1]{\int_0^T \int_{\Omega} #1 \ \dxdt}

\newcommand{\D}{{\rm d}}

\newcommand{\R}{\mathbb{R}}

\newcommand{\calF}{\mathcal{F}}
\newcommand{\calP}{\mathcal{P}}

\definecolor{Cgrey}{rgb}{0.85,0.85,0.85}
\definecolor{Cblue}{rgb}{0.50,0.85,0.85}
\definecolor{Cred}{rgb}{1,0,0}
\definecolor{fancy}{rgb}{0.10,0.85,0.10}
\definecolor{forestgreen}{rgb}{0.13, 0.55, 0.13}

\date{}
\newcommand{\pd}{\partial}
\newcommand{\Hc}{P}
\newcommand{\eps}{\varepsilon}
\newcommand{\faces}{\mathcal{E}}

\newcommand{\facesK}{\faces(K)}

\newcommand{\facesint}{\faces}

\newcommand{\Qh}{ Q_h}
\newcommand{\vQh}{ \mathfrak{Q}_h}

\newcommand{\bVh}{ \mathbf{V}_h}

\newcommand{\Xh}{ X_h}
\newcommand{\Wh}{ W_h}

\date{}

\begin{document}


\title{Numerical analysis of a model of two phase compressible fluid flow}

\author{Eduard Feireisl$^{1}$\and M\u ad\u alina Petcu$^{2,3,4}$ \and Bangwei She$^{1,5}$  }


\maketitle

\medskip

\centerline{${}^1$Institute of Mathematics of the Academy of Sciences of the Czech Republic}
\centerline{\v Zitn\' a 25, CZ-115 67 Praha 1, Czech Republic}

\bigskip

\centerline{${}^2$Laboratoire de Math\' ematiques et Applications, UMR CNRS 7348 - SP2MI}

\centerline{Universit\' e de Poitiers, Boulevard Marie et Pierre Curie - T\' el\' eport 2}

\centerline{86962 Chasseneuil, Futuroscope Cedex,
France}

\centerline{${}^3$The Institute of Mathematics of the Romanian Academy, Bucharest, Romania}

\centerline{and}

\centerline{${}^4$The Institute of Statistics and Applied Mathematics of the Romanian Academy, Bucharest, Romania}

\bigskip

\centerline{${}^5$Department of Mathematical Analysis, Charles University}
\centerline{Sokolovsk\' a 83, CZ-186 75 Praha 8, Czech Republic}

\begin{abstract}

We consider a model of a binary mixture of two immiscible compressible fluids. We  propose a numerical scheme and discuss its basic properties: Stability, consistency, convergence. The convergence is established via the method of generalized weak solutions combined with the weak--strong uniqueness principle.

\end{abstract}



\section{Introduction}
\label{I}

We consider a binary mixture of two immiscible compressible fluids. There are several ways how to model such a system. Here, we consider the model introduced in \cite{FePePr} based on the phase field approach:

\begin{equation} \label{i1}
\begin{split}
\partial_t \vr + \Div (\vr \vu) &= 0,\\
\partial_t (\vr \vu) + \Div (\vr \vu \otimes \vu) + \Grad p(\vr) &= 
\Div \mathbb{S}({\Grad} \vu) - \Div \left( \Grad \cc \otimes \Grad \cc - \frac{1}{2} |\Grad \cc|^2 \mathbb{I} \right)  + \Grad F(\cc)    ,\\ 
\partial_t \cc + \vc{u} \cdot \Grad \cc &=  \mu,              
\end{split} 
\end{equation}
where 
\begin{equation} \label{i2}
\begin{split}
\mu&= \Del \cc - F'(\cc),\\ 
\mathbb{S}( \Grad \vu) &= \nu \left( \Grad \vu + \Grad^t \vu - \frac{2}{d} \Div \vu \mathbb{I} \right) + 
\lambda \Div \vu \mathbb{I},\ \nu > 0, \ \lambda \geq 0.
\end{split} 
\end{equation}
The system \eqref{i1}, \eqref{i2} is a variant of the general phase field approach, where the density of the mixture is 
represented by a single scalar function $\vr$, the joint velocity is $\vu$, and the concentration difference of the two phases is the order 
parameter $\cc$, the evolution of which is governed by the Allen--Cahn equation.

The energy of the system is 
\[
E(\vr, \vu, \cc) = \frac{1}{2} \vr |\vu|^2 + \frac{1}{2} |\Grad \cc |^2 + P(\vr) + F(\cc), 
\]
where $P$ is the pressure potential  and $F$ is the Ginzburg-Landau potential. The pressure potential 
is related to the pressure $p$ as
\[
P'(\vr) \vr - P(\vr) = p(\vr).
\] 
Moreover, we suppose 
\begin{equation} \label{H3}
\begin{split}
&p \in C[0, \infty) \cap C^\infty(0,\infty),\\ 
&p'(\vr) > 0 \ \mbox{for}\ \vr > 0,\ \liminf_{\vr \to \infty} p'(\vr) > 0,\ 
p(\vr) \leq c \left(1 + P (\vr) \right) \ \mbox{for all}\ \vr \geq 0. 
\end{split}
\end{equation}  
The Ginzburg--Landau potential takes the form 
\begin{equation}\label{F}
 F(\cc)= 
\left\{\begin{array}{ll}
&  (\cc+1)^2, \quad{\text {on}} \quad \cc \leq-1,\\
&  \frac14  (\cc^2-1)^2, \quad{\text{ on}} \quad -1 \leq \cc \leq 1,\\
&  (\cc-1)^2, \quad{\text {on}} \quad \cc >1.
\end{array}\right.
\end{equation}
Note that $F$ coincides with the more standard double well potential $F(\cc)=\frac14 (\cc^2-1)^2$ 
in the physically relevant area $\cc \in [-1,1]$. All results of this paper remain valid 
for $F$ of the form $F(\cc)=\lambda \cc^2+ W(\cc)$, $\lambda > 0$, with $W\in \mathcal{C}^2 \cap W^{2, \infty}(\R)$ such that $W$ and $W'$ are (globally) Lipschitz functions.

We consider either the simplified periodic boundary conditions, where the physical domain can be identified with the flat torus 
\begin{subequations}\label{ibc}
\begin{equation} \label{i3}
\Omega = \mathcal{T}^{d} = \left( [-1,1]|_{\{ -1, 1 \}} \right)^{d},\ {d}=2,3, 
\end{equation}
or the Dirichlet boundary conditions 
\begin{equation} \label{i4}
\vu|_{\partial \Omega} = 0,\ \cc|_{\partial \Omega} = 0,\  
\Omega \subset \R^{d} \ \mbox{a bounded domain.}
\end{equation}
\end{subequations}

In comparison with more complex models proposed by Blesgen \cite{Blesgen} or Anderson et al. \cite{AnFaWh}, the present model 
is much simpler to facilitate numerical analysis. To the best of our knowledge, this is the first attempt in the context of mixtures of \emph{compressible} fluids.

If the boundary conditions \eqref{ibc} are imposed, the total energy is a Lyapunov function, 
\begin{equation} \label{i5}
\frac{{\rm d}}{{\rm d}t} \intO{ \left[ \frac{1}{2} \vr |\vu|^2 + \frac{1}{2} |\Grad \cc |^2 + P(\vr) + F(\cc) \right] } 
+ \intO{  \mathbb{S}(\Grad \vu) : \Grad \vu } + \intO{ |\Del \cc - F'(\cc)|^2 } = 0.
\end{equation}

Our goal in the present paper is 
\begin{itemize}

\item to propose a numerical scheme to solve \eqref{i1}, \eqref{i2}, with the boundary conditions \eqref{ibc};
\item to show stability estimates and consistency of the scheme; 
\item to show convergence of numerical approximations to a regular solution as long as it exists.

\end{itemize}

The strategy is to identify a large class of generalized solutions to the problem that goes beyond the standard framework of weak solutions. These are the so called \emph{dissipative weak solutions} similar to those introduced in the monograph \cite{FeLMMiSh}.  Although 
dissipative  weak  solutions are more general objects than the weak solutions, they still comply with the weak--strong uniqueness principle. 
A dissipative weak solution coincides with the strong solution originating from the same initial data. Then we show that any sequence 
of numerical solutions that is stable and consistent converges to a dissipative  weak solution. Applying the weak--strong uniqueness principle, we finally prove unconditional convergence to the strong solution, provided the latter exists.

\section{Dissipative weak solutions}
\label{D}

In accordance with the choice of the boundary conditions \eqref{ibc}, we will use the same symbol $\Omega$ to denote 
either the flat torus $\mathcal{T}^d$ or a bounded domain in $\R^d$. The symbol $\mathcal{M}(\Omega; X)$ denotes the set of (Radon) 
measures on $\Omega$ ranging in a (finite--dimensional) space $X$, $\mathcal{M}^+$ is the cone of non--negative scalar--valued measures.

The anticipated integrability properties of dissipative solutions are in agreement with the energy balance \eqref{i5}: 
\begin{equation} \label{D1}
\begin{split}
\sqrt{\vr} \vu &\in L^\infty(0,T; L^2(\Omega; \R^d)),\ \cc \in L^\infty(0,T; W^{1,2}(\Omega)),\ 
P(\vr) \in L^\infty(0,T; L^1(\Omega))\\ 
\vu &\in L^2(0,T; W^{1,2}(\Omega; \R^d)),\ \cc \in L^2(0,T; W^{2,2}(\Omega)).
\end{split}
\end{equation}

\begin{Definition} [Dissipative  weak  solution] \label{DD1}

A trio $\{ \vr, \vu, \cc \}$ is called \emph{dissipative weak solution} of the problem \eqref{i1}, \eqref{i2} in 
$(0,T) \times \Omega$, with the boundary 
conditions \eqref{i3} (or \eqref{i4}) and the initial conditions 
\[
\vr(0, \cdot) = \vr_0,\ \vr \vu(0, \cdot) = (\vr \vu)_0, \ \cc(0, \cdot) = \cc_0,
\]
if the following is satisfied:
\begin{itemize}
\item {\bf Regularity.} The solution belongs to the class \eqref{D1}. Moreover, 
\[
\vr \geq 0 \ \mbox{a.a. in}\ (0,T) \times \Omega.
\]

\item {\bf Equation of continuity.} 

\begin{equation} \label{D2}
\int_0^T \intO{ \Big[ \vr \partial_t \varphi + \vr \vu \cdot \Grad \varphi \Big] } \dt = 
- \intO{ \vr_0 \varphi(0, \cdot) }
\end{equation}
for any $\varphi \in C^1_c([0,T) \times  \Ov{\Omega} )$.

\item {\bf Momentum equation.} 

\begin{equation} \label{D3}
\begin{split}
\int_0^T &\intO{ \Big[ \vr \vu \cdot \partial_t \bfphi + \vr \vu \otimes \vu : \Grad \bfphi +  p(\vr) 
 \Div \bfphi \Big] } \dt 
\\ &= \int_0^T \intO{ \left[ \mathbb{S}(\Grad \vu) : \Grad \bfphi  + (\Lap \cc -F'(\cc) ) \Grad \cc \cdot   \bfphi  \right] } \dt\\
&- \intO{ (\vr \vu)_0 \cdot \bfphi(0, \cdot) } + \int_0^T \int_{\Omega} \Grad \bfphi : \D \mathfrak{R}(t) \dt
\end{split}
\end{equation}
for any $\bfphi \in C^1_c([0,T) \times \Omega; \R^d)$, where 
\[
\mathfrak{R} \in L^\infty(0,T; \mathcal{M} (\Omega; \R^{d \times d}_{\rm sym})). 
\]
In the case of the Dirichlet boundary conditions \eqref{i4}, we also require 
\[
\vu \in L^2(0,T; W^{1,2}_0 (\Omega; \R^d)).
\]

\item {\bf Allen--Cahn equation for the concentration difference.}

\begin{equation} \label{D4}
\begin{split}
\int_0^T \intO{ \Big[ \cc \partial_t \varphi - \vu \cdot \Grad \cc \varphi \Big] } \dt &= - \int_0^T \intO{ \mu \varphi } \dt - 
\intO{ \cc_0 \varphi (0, \cdot) },\\ 
\mu &= \Del \cc - F'(\cc)
\end{split}
\end{equation} 
for any $\varphi \in C^1_c ([0,T) \times \Omega)$. If \eqref{i4} is imposed, we require 
\[
\cc \in L^\infty(0,T; W^{1,2}_0(\Omega)).
\]

\item {\bf Energy balance.} 

\begin{equation} \label{D5}
\begin{split}
&\intO{ \left[ \frac{1}{2} \vr |\vu|^2 + \frac{1}{2} |\Grad \cc |^2 + P(\vr) + F(\cc) \right](\tau, \cdot) }\\ 
&+ \int_0^\tau \intO{ \mathbb{S}(\Grad \vu) : \Grad \vu } + \intO{ |\Del \cc - F'(\cc)|^2 } \dt 
+ \int_\Omega \D \mathfrak{E} (\tau) \\
&\leq \intO{ \left[ \frac{1}{2} \frac{ |(\vr \vu)_0|^2 }{\vr_0} + \frac{1}{2} |\Grad \cc_0 |^2 + P(\vr_0) + F(\cc_0) \right] }
\end{split}
\end{equation}
for a.a. $\tau \in (0,T)$, where 
\[
\mathfrak{E} \in L^\infty(0,T; \mathcal{M}^+ ({\Ov{\Omega} })).
\] 

\item {\bf Defect compatibility.} 

\begin{equation} \label{D6}
| \mathfrak{R} (\tau) | \aleq \mathfrak{E}(\tau) 
\end{equation}
for a.a. $\tau \in (0,T)$.

\end{itemize}

\end{Definition}

\begin{Remark} \label{DR1}

The inequality \eqref{D6} can be interpreted that there exists a constant $\Lambda > 0$ such that  
\[
\Big( \Lambda \mathfrak{E} \mathbb{I} \pm \mathfrak{R} \Big) (\tau) \geq 0,
\]
meaning 
\[
\int_\Omega \phi \xi \otimes \xi : \D \left(
\Lambda \mathfrak{E}(\tau) \pm \mathfrak{R} \right) (\tau) \geq 0
\]
for any $\phi \in C_c(\Omega)$, $\phi \geq 0$, $\xi \in \R^d$. 

\end{Remark}

The measure $\mathfrak{R}$ can be viewed as the sum of a concentration and oscillation defects related to the non--linearities in 
the momentum balance. Similarly, the measure $\mathfrak{E}$ results from the defect in due to possible ``anomalous'' energy dissipation. The compatibility property \eqref{D6} is absolutely crucial for the weak--strong uniqueness principle. A more elaborate 
discussion concerning this approach can be found in the monograph \cite{FeLMMiSh}.

\section{Relative energy inequality}\label{REI}

The {\it relative energy} is a non-negative quantity that represents a ``distance'' between 
a dissipative weak solution in the sense of Definition \ref{DD1} and any trio of suitable smooth test functions. 
It can be also interpreted as the Bregman distance generated by the energy functional, cf. Sprung \cite{Sprung}.  
In this section we derive a differential inequality satisfied by the relative energy. Later we use it to evaluate the distance between a dissipative weak solution for the problem \eqref{i1}-\eqref{i2} and a more regular one. This technique, introduced by Dafermos in \cite{Daf4}, was largely used in order to prove weak-strong uniqueness for the solutions  of different types of partial differential equations (see for example \cite{FeJiNo,Germ,Wi}) as well as to study certain singular limits as for example incompressible, inviscid limits of compressible, viscous fluids (see \cite{feinov,Ma,SaR}).

Let $R$, $R>0$, $\vU$, and $\CC$ be arbitrary continuously differentiable functions. We define the \emph{relative energy} as 
\begin{align*}
\mathcal{E} &\left( \vr, \vu,\cc \ \Big|  R, \vc{U}, \CC \right) \\ &= 
\intTN{ \left[ \frac{1}{2} \vr |\vu - \vc{U} |^2 +(\cc-\CC)^2  + \frac{1}{2} |\Grad \cc - \Grad \CC|^2 
+ P(\vr) - P'(R)(\vr - R) - P(R) \right] }.
\end{align*}
In contrast with its counterpart introduced in \cite{FePePr}, the present relative energy is augmented 
by the term $(\cc-\CC)^2$ to compensate the absence of Poincar\' e inequality in the periodic case.  

In order to facilitate the computations, we can decompose the relative energy $\mathcal{E}$ as follows:
\[
\mathcal{E} \left( \vr, \vu,\cc \ \Big| R, \vc{U}, \CC \right) =  \sum_{j=1}^5 I_j,
\]
with 
\begin{align*}
&I_1 = \intTN{ \left[ \frac{1}{2} \vr |\vu|^2 + P(\vr) + \cc^2+ \frac{1}{2} |\Grad \cc|^2 \right] },
\quad  
I_2 = \intTN{ \vr \left[ \frac{1}{2} |\vc{U}|^2  - P'(R) \right] },
\\&
I_3 = - \intTN{ \vr \vu \cdot \vc{U} },\quad  
I_4 = - \intTN{ \Grad \cc \cdot \Grad \CC +2 \cc \CC},
\\ &
I_5 = \intTN{ \left[\CC^2+ \frac{1}{2} |\Grad \CC|^2 + p(R) \right] }, {\textrm{ where }}\ p(R)
= P'(R) R - P(R).
\end{align*}

Let us now state and prove the main result of this section:
\bTheorem{TRE}
Let $(\vr, \vu, \cc)$ be a dissipative weak solution to problem \eqref{i1}-\eqref{i2}, \eqref{ibc}, in the sense specified in Definition \ref{DD1}. 

Then, for any continuously differentiable functions satisfying 
$$
R \in C^1([0,T] \times \Omega),\ R > 0,\ 
\vU \in C^1([0,T] \times \Omega; \R^d),\ 
\ \CC, \Grad \CC, \Del \CC \in C^1([0,T] \times \Omega),
$$
and, in the case of the Dirichlet boundary conditions, also
\[
\vU|_{\partial \Omega} = 0,\ \CC|_{\partial \Omega} = 0,
\]
the following relative energy inequality holds:
\begin{equation}\label{rel_ener}
\begin{split}
&\left[ \mathcal{E}\left( \vr, \cc, \vu \ \Big| R, \CC, \vc{U} \right) \right]_{t = 0}^{ t = \tau} +  \int_0^\tau \intTN{ \Big[ \mathbb{S}(\Grad \vu-\Grad \vc{U}) : \left( \Grad \vu 
- \Grad \bf{U} \right)  +  \mu^2 \Big] } \dt + \int_{ \Ov{\Omega} } {\rm d} \EE(\tau)
\\& 
\leq - \int_0^\tau \intTN{ \mu(\Del \cc -\mu -2 \cc)} \dt
-  \int_0^\tau \intTN{ (\Div \vu -\Div \vc{U} )W(\cc)} \dt 
\\&+\int_0^\tau \intTN{ \vr(\vc{U}- \vu)\cdot \partial_t \vc{U}} \dt
+ \int_0^\tau \intTN{ (R-\vr)\partial_t P'(R) } \dt
\\&
- \int_0^\tau \intTN{\vr \vu \cdot \Grad P'(R)}\dt 
+ \int_0^\tau \intTN{ \vr \vu \cdot  \Grad \vc{U} \cdot (\vc{U}-\vu) } \dt 
\\&
- \int_0^\tau \intTN{(p(\vr )-\cc^2) \Div \vc{U}} \dt
- \int_0^\tau \intTN{ \left( \Grad c \otimes \Grad c - \frac{1}{2} |\Grad c|^2 \mathbb{I} \right) : \Grad {\bf{U}} } \dt
\\&
+ \int_0^\tau \int_{\Omega}{  \Grad {\bf{U}} : {\rm d}\RR(t) } \dt
+ \int_0^\tau \intTN{ (\mu - \vu \cdot \Grad \cc) (\Del \CC - 2 \CC)} \dt 
\\&
+ \int_0^\tau \intTN{ (\cc-\CC) (\Del \CC_t- 2\CC_t)} \dt 
- \int_0^\tau \intTN{ \Big[ \mathbb{S}(\Grad \vc{U}) : \left( \Grad \vu 
- \Grad \bf{U} \right) \Big] } \dt
\end{split}
\end{equation}
for a.a. $\tau \in (0,T)$.
\eT
\begin{proof}
Since a dissipative weak solution satisfies the energy inequality \eqref{D5}, the term $I_1$ from the relative energy is estimated as follows:
\begin{equation} \label{R1}
\begin{split}
\left[ I_1 \right]_{t = 0}^{t = \tau} &= 
\left[ \intTN{ \left[  \frac{1}{2} \vr |\vu|^2 + P(\vr) + \cc^2 + W(\cc)
+ \frac{1}{2} |\Grad \cc|^2 \right] } \right]_{t = 0}^{t = \tau} - 
\left[ \intO{ W (\cc) } \right]_{t=0}^{t = \tau}
\\&
\leq -\int_{\Ov{\Omega}}{{\rm d}\EE(\tau)}  - \int_0^\tau \intTN{ \Big[ \mathbb{S}(\Grad \vu) : \Grad \vu +  \mu^2 \Big] } 
- \left[ \intO{ W(\cc) } \right]_{t=0}^{t = \tau},
\end{split}
\end{equation}
 where $W(\cc)=F(\cc)-\cc^2$.

For the last term in \eqref{R1}, we use the Allen-Cahn equation for the concentration $\cc$ and we get:
\begin{equation} \label{R1b}
\begin{split}
& \left[ \intO{ W(\cc) } \right]_{t=0}^{t = \tau}=-\int_0^\tau \intO{\partial_t W(\cc)} \dt=-\int_0^\tau \intO{\partial_t \cc W'(\cc)} \dt\\
& =- \int_0^\tau \intO{(\mu - \vu \cdot \Grad \cc)  W'(\cc)} \dt\\
&=- \int_0^\tau \intO{\mu(\Lap  \cc - \mu -2 \cc)} \dt+ \int_0^\tau \intO{\vu \cdot\Grad W(\cc)} \dt\\
&=- \int_0^\tau \intO{\mu(\Lap  \cc - \mu -2 \cc)} \dt- \int_0^\tau \intO{\Div \vu\  W(\cc)} \dt.
\end{split}
\end{equation}

The estimates for $I_2$ are obtained testing the continuity equation by $\varphi=\dfrac12 |\vc{U}|^2-P'(R)$. Remark that the choice of $\varphi$  is possible thanks to the regularity of the functions $R$ and $\vc{U}$. We obtain:
\begin{equation} \label{R4}
\begin{split}
[ I_2 ]_{t = 0}^{ t = \tau} &= \int_0^\tau \intTN{ \left[ \vr \Big(
\vc{U} \cdot \partial_t \vc{U}  -  \partial_t P'(R) \Big)   
\right] } \dt \\&+ \int_0^\tau \intTN{ \left[ \vr \vu \cdot \Big(  \Grad^t \vc{U} \cdot  \vc{U}   
 - \Grad P'(R) \Big)   
\right] } \dt. 
\end{split}
\end{equation}

Adding the inequalities for $I_1$ and $I_2$, we obtain:
\begin{equation} \label{R5}
\begin{split}
[ I_1 +& I_2 ]_{t = 0}^{ t = \tau} + \int_{\Ov{\Omega}} {\rm d}\EE(\tau)+  \int_0^\tau \intTN{ \Big[ \mathbb{S}(\Grad \vu) : \Grad \vu +  \mu^2 \Big] }
\\&
\leq - \int_0^\tau \intTN{  \mu (\Lap  \cc - \mu - 2\cc) } \dt
-  \int_0^\tau \intTN{ \Div \vu \ W(\cc)} \dt 
\\&
+  \int_0^\tau \intTN{(\vr \vc{U} \cdot \partial_t \vc{U} -\vr \partial_t P'(R) )} \dt + \int_0^\tau \intTN{\vr \vu \cdot (  \Grad^t \vc{U} \cdot  \vc{U} - \Grad P'(R) )} \dt. 
\end{split}
\end{equation}

We also test the momentum equation by $\vc{U}$ and get:
\begin{equation}\label{R3}
\begin{split}
[ I_3 ]_{t = 0}^{t = \tau} = 
&-
\int_0^\tau \intTN{ \left[ \vr \vu \cdot \partial_t {\bf{U}} + \vr \vu \cdot  \Grad \bf{U} \cdot \vu  +\big( p(\vr)-\cc^2-W(\cc)\big) \Div \bf{U} \right] } 
\dt 
\\&
+ \int_0^\tau \intTN{ \mathbb{S}(\Grad \vu) : \Grad{ \bf{U}} } \dt  + \int_0^\tau \int_{\Omega}{ \Grad \vc{U} : {\rm d}\ \RR(t)}\dt 
\\&
- \int_0^\tau \intTN{ \left( \Grad c \otimes \Grad c - \frac{1}{2} |\Grad c|^2 \mathbb{I} \right) : \Grad {\bf U} } \dt.
\end{split}
\end{equation}
Summing \eqref{R5} and \eqref{R3}, we have:
\begin{equation} \label{R7}
\begin{split}
[ I_1& + I_2 + I_3 ]_{t = 0}^{ t = \tau} + \int_{\Ov{\Omega}} {\rm d} \EE(\tau) +  \int_0^\tau \intTN{ \Big[ \mathbb{S}(\Grad \vu) : \left( \Grad \vu 
- \Grad \bf{U} \right)  +  \mu^2 \Big] } \dt \\
&\leq - \int_0^\tau \intTN{ \mu(\Del \cc -\mu -2\cc)  } \dt  - \int_0^\tau \intTN{ (\Div \vu - \Div \vc{U})W(\cc)}\dt\\
&+  \int_0^\tau \intTN{ \left[ \vr (\vc{U} -\vc{u})  \cdot \partial_t \vc{U}  + \vr \vu \cdot  \Grad \vc{U}\cdot (\vc{U} - \vc{u})  \right] } \dt \\
&- \int_0^\tau \intTN{ \left[ \vr \partial_t P'(R) + \vr \vu \cdot
\Grad P'(R) \right] } \dt 
- \int_0^\tau \intTN{( p(\vr)-\cc^2) \Div \vc{U} } \dt\\
&- \int_0^\tau \intTN{ \left( \Grad c \otimes \Grad c - \frac{1}{2} |\Grad c|^2 \mathbb{I} \right) : \Grad {\bf{U}} } \dt
+ \int_0^\tau  \int_{\Omega} \Grad {\bf{U}}: {\rm d} \RR(t)  \dt. 
\end{split}
\end{equation}
For the last two terms in the relative energy, we use the Allen-Cahn equation. Thus:
\begin{equation} \label{R6}
\begin{split}
[I_4]_{t = 0}^{t = \tau} &= - \left[ \intTN{ \Grad \cc \cdot \Grad \CC+ 2 \cc \CC} \right]_{t = 0}^{t = \tau} = \left[ \intTN{  \left[\Del \CC-2\CC \right] \cc} 
\right]_{t = 0}^{t = \tau} \\ &= 
\int_0^\tau \intTN{ \partial_t \left[\Del \CC -2 \CC\right]\cc } \dt + 
\int_0^\tau \intTN{ \left[\Del \CC-2\CC\right] \partial_t \cc } \dt  \\
&= \int_0^\tau \intO{ \partial_t \left[\Del \CC-2\CC\right] \cc } \dt 
+ \int_0^\tau \intTN{ \left[\Del \CC-2\CC \right] \left[ \mu - \vu \cdot \Grad \cc \right] } \dt, 
\end{split}
\end{equation}
and
\begin{equation} \label{R6b}
\begin{split}
[I_5]_{t = 0}^{t = \tau} &= -  
\int_0^\tau \intTN{ ( \CC \ \Lap  \CC_t -2 \CC \ \CC_t)} \dt + 
\int_0^\tau \intTN{R \partial_t P'(R) } \dt.  
\end{split}
\end{equation}
 thanks to the equality $\pd_t p(R) =\pd_t(P'(R)R-P(R)) = R \pd_tP'(R)$. 

Adding \eqref{R6} and \eqref{R6b} to \eqref{R7}, we obtain the desired relative energy inequality.

\end{proof}

\section{Weak-strong uniqueness}\label{WSU}

In this section we prove that a dissipative weak solution and a strong solution for the compressible Navier-Stokes-Allen-Cahn problem \eqref{i1}--\eqref{i2}, \eqref{i3}/\eqref{i4}, both emanating from the same initial data, coincide on the life span of the strong solution. More exactly, we prove the following result:
\bTheorem{T_wsu} Let the initial data $(\vr_0, \vm_0, \cc_0)$ be given such that the initial energy is finite
$$
\intO{ \left[ \frac{1}{2} \frac{ |\vm_0|^2 }{\vr_0} + \frac{1}{2} |\Grad \cc_0 |^2 + P(\vr_0) + F(\cc_0) \right] }< \infty.
$$
Let $(\vr, \vu, \cc)$ be a dissipative weak solution of the problem \eqref{i1}--\eqref{i2},  \eqref{i3}/\eqref{i4} in $(0, T)\times \Omega$ in the sense of Definition~\ref{DD1}, with the initial data $(\vr_0, \vm_0, \cc_0)$. Suppose that $(R, \vc{U}, \CC)$ is a strong solution of the same problem belonging to the class:
\begin{equation}\label{cond_strong}
\begin{split}
& \inf_{(0, T)\times \Omega}R >0, \, R \in \mathcal{C}^1([0, T] \times \Ov \Omega),\\
& \vc{U} \in C^1([0, T] \times \Ov \Omega; \R^N), \> \Div \mathbb{S}(\Grad {\bf U}) \in C([0, T] \times \Ov \Omega; \R^N),\\
& \CC, \ \Grad \CC, \ \Del \CC \in  \mathcal{C}^1([0, T] \times \Ov \Omega),
\end{split}
\end{equation}
and such that
$$
R(0, \cdot)=\vr_0, \, R(0, \cdot) \vc{U}(0, \cdot)= \vm_0, \, \CC(0, \cdot)=\cc_0.
$$

Then
$$
\vr=R, \, \vu=\vc{U}, \, \cc=\CC {\textrm{ in } } (0, T)\times \Omega,
$$
and
$$
\EE=\RR=0.
$$
\eT
\begin{proof}
The idea of the proof is to test the relative energy inequality by the strong solution and use the Gronwall inequality in order to obtain the desired result. Let us proceed by considering the following terms from the relative energy inequality:
\begin{align*}
J&=\int_0^\tau \intTN{ \left[ \vr (\vc{U} -\vc{u})  \cdot \partial_t \vc{U}  + \vr \vu \cdot   \Grad \vc{U} \cdot (\vc{U} - \vc{u})  \right] } \dt - \int_0^\tau \intTN{\mathbb{S}(\Grad {\bf U}) : \left( \Grad \vu 
- \Grad \bf{U} \right) } \dt  \\
&= \int_0^\tau \intTN{ \left[ \frac{\vr}{R} (\vc{U} -\vc{u})  \cdot R
\partial_t \vc{U} \right] } \dt -\int_0^\tau \intTN{ \mathbb{S}(\Grad {\bf U}) : \left(\Grad \vu 
- \Grad \bf{U} \right) } \dt  \\
&+ \int_0^\tau \intTN{ \vr  \vc{U} \cdot   \Grad \vc{U} \cdot (\vc{U} - \vc{u})  } \dt +\int_0^\tau \intTN{  \vr (\vu- \vc{U} ) \cdot   \Grad \vc{U} \cdot (\vc{U} - \vc{u}) } \dt\\
&=\sum_{i=1}^4J_i.
\end{align*}

For $J_1+J_2$ we use the equation of momentum and get:
\begin{equation}\begin{split}
J_1+J_2= &- \int_0^\tau \intTN{ \frac{\vr}{R} (\vc{U} -\vc{u}) \cdot \left[  \Div \mathbb{S}(\Grad \vc{U}) - \Div \left( \Grad \CC \otimes \Grad \CC - \frac{1}{2} |\Grad \CC|^2 \mathbb{I} \right)  -\vc{U} \partial_t R\right]}\dt\\
& -\int_0^\tau \intTN{ \mathbb{S}(\Grad {\bf U}) : \left(\Grad \vu 
- \Grad \bf{U} \right) } \dt  -\int_0^\tau \intTN{ \frac{\vr}{R} (\vc{U} -\vc{u}) \cdot \Div \left( R \vc{U} \otimes \vc{U}\right)}\dt \\
&- \int_0^\tau \intTN{\frac{\vr}{R} (\vc{U} -\vc{u}) \cdot \Grad\left(p(R)-\CC^2 - W(\CC)\right)}\dt\\
= &- \int_0^\tau \intTN{ \frac{\vr}{R} (\vc{U} -\vc{u}) \cdot \left[\Grad \CC \Del \CC +R   \Grad^t \vc{U} \cdot  \vc{U} \right]}\dt\\
&- \int_0^\tau \intTN{ \frac{\vr-R}{R} (\vc{U} -\vc{u}) \cdot  \Div \mathbb{S}(\Grad \vc{U})}\\
&- \int_0^\tau \intTN{\frac{\vr}{R} (\vc{U} -\vc{u}) \cdot \Grad\left(p(R)-\CC^2 - W(\CC)\right)}\dt.
\end{split}\end{equation}
which implies that:
\begin{equation}\label{j123}
\begin{split}
&J_1+J_2 +J_3=\\
&- \int_0^\tau \intTN{ \frac{\vr}{R} (\vc{U} -\vc{u}) \cdot \Grad \CC \Del \CC}\dt
- \int_0^\tau \intTN{ \frac{\vr-R}{R} (\vc{U} -\vc{u}) \cdot  \Div \mathbb{S}(\Grad \vc{U})}\dt\\
&- \int_0^\tau \intTN{\frac{\vr}{R} (\vc{U} -\vc{u}) \cdot \Grad\left(p(R)-\CC^2 - W(\CC)\right)}\dt.
\end{split}\end{equation}
The last term in $J$ can be easily estimated as follows:
$$
|J_4|\leq \int_0^\tau |\Grad \vc{U}|_{L^\infty(\Omega)} \mathcal{E}(\vr, \vu, \cc| R, \vc{U}, \CC)(t) \dt.
$$
Using exactly the same arguments as in \cite{FeJiNo}, we also estimate:
\begin{equation}\begin{split}
\big|\int_0^\tau \intTN{ \frac{\vr-R}{R} (\vc{U} -\vc{u}) \cdot  \Div \mathbb{S}(\Grad \vc{U})}\dt\big|\leq & \delta\int_0^\tau \intO{ \left( \mathbb{S}(\Grad \vu) - \mathbb{S}(\Grad \vU) \right): \left( \Grad \vu 
- \Grad \bf{U} \right) }\dt\\
&+ c(\delta) \int_0^\tau \mathcal{E}(\vr, \vu, \cc| R, \vc{U}, \CC)(t) \dt,
\end{split}
\end{equation}
for any $\delta>0$, where $c(\delta)$ is a positive constant depending on $\delta$ and on certain norms of $R$ and ${\bf U}$. The estimate is based on the Korn-Poincar\'e inequality
{(see e.g.\ \cite{feinov})}:
\begin{equation} \label{KoPo}
\intO{ \left| \vu - \vU \right|^2 + \left| \Grad (\vu - \vU) \right|^2 } 
\leq c_{kp} \intO{ \left[\left( \mathbb{S}(\Grad \vu) - \mathbb{S}(\Grad \vU) \right): \left( \Grad \vu 
- \Grad \bf{U} \right) +\vr \left| \vu - \vU \right|^2\right]}.  
\end{equation}

The last term in \eqref{j123} we split it:
\begin{equation}\begin{split}
&\int_0^\tau \intTN{\frac{\vr}{R} (\vc{U} -\vc{u}) \cdot \Grad\left(p(R)-\CC^2 - W(\CC)\right)}\dt=\int_0^\tau \intTN{ (\vc{U} -\vc{u}) \cdot \Grad\left(p(R)-\CC^2 - W(\CC)\right)}\dt\\
&+\int_0^\tau \intTN{\frac{\vr-R}{R} (\vc{U} -\vc{u}) \cdot \Grad\left(p(R)-\CC^2 - W(\CC)\right)}\dt.
\end{split}\end{equation}
Using the fact that $R$ and $\CC$ are regular enough, we can bound
\begin{equation}\label{drR}
  \big|\int_0^\tau \intTN{\frac{\vr-R}{R} (\vc{U} -\vc{u}) \cdot \Grad\left(p(R)-\CC^2 - W(\CC)\right)}\dt \big| \leq c(R, \CC) \int_0^\tau \intTN{|\vr-R| |\vc{U} -\vc{u}| }\dt,
\end{equation}
where $c(R, \CC)$ is a positive constant depending on the norms of $R$ and $\CC$.

 Using the same arguments as in \cite{FeJiNo} and \cite{FePePr}, we now introduce the following cut-off function:
 \[\Psi \in \DC(0, \infty), \,
0 \leq \Psi \leq 1, \ \Psi \equiv 1 \ \mbox{in} \ [\delta, \frac{1}{\delta}], 
\]
where $\delta$ is chosen so small that 
\[
R(t,x) \in [2 \delta, \frac{1}{2 \delta}]\ \mbox{for all}\ (t,x) \in [0,T] \times \Ov{\Omega}.
\]
For any function $h \in L^1((0,T) \times \Omega)$, we set the following splitting
\[
h = h_{{\rm ess}} + h_{\rm res},\  h_{{\rm ess}} = \Psi (\vr) h,\ 
h_{\rm res} = (1 - \Psi (\vr)) h.
\]
We can thus continue to estimate in \eqref{drR} as:
\begin{equation}\label{}
   \int_0^\tau \intTN{|\vr-R| |\vc{U} -\vc{u}| }\dt \leq  \int_0^\tau \intTN{|\vr-R|_{\rm ess} |\vc{U} -\vc{u}|_{\rm ess} }\dt +  \int_0^\tau \intTN{|\vr-R|_{\rm res} |\vc{U} -\vc{u}| }\dt,
\end{equation}
where we used the fact that:
\begin{equation}	\label{yng1}
 \left| [ \vr - r ]_{\rm ess} \right| \ | \vc{U} -\vc{u} |
 = \sqrt{ [ \vr - r ]^2_{\rm ess} }	
	\sqrt{ | \vc{U} -\vc{u} |^2_{\rm ess}}.
\end{equation}
It can be easily checked that 
\begin{equation}	\label{scFe}
P(\vr) - P(r)(\vr-r) - P(r) \ageq ( \vr - r )^2_{\rm ess}
+ (1 + \vr)_{\rm res} \,.
\end{equation}
We can thus write
\begin{equation} \label{ineq}
\begin{split}
\mathcal{E} \left( \vr, \vu, c \ \Big| \ r , \vc{U}, C \right) \ageq
\intO{ \left( |\vu - \vc{U} |^2_{\rm ess} + [ \vr - r ]^2_{\rm ess} + 1_{\rm res} + \vr_{\rm res} \right)},  
\end{split} 
\end{equation}
which allows us to conclude that:
$$
\int_0^\tau \intTN{|\vr-R|_{\rm ess} |\vc{U} -\vc{u}|_{\rm ess} }\dt \aleq \mathcal{E} \left( \vr, \vu, c \ \Big| \ r , \vc{U}, C \right).
$$
We also know that:
\begin{equation}	\label{yng2}
 \left| [ \vr - r ]_{\rm res} \right| \ | \vc{U} -\vc{u} |
\aleq 
1_{\rm res} | \vc{U} -\vc{u} | + \sqrt{\vr_{\rm res}}
\sqrt{\vr} | \vc{U} -\vc{u} |,
\end{equation}
which implies: 
\begin{align*}
\int_0^\tau &\intTN{  \left| [ \vr - r ]_{\rm res} \right| \ | \vc{U} -\vc{u} |  } \dt 
\\ &\leq c(\delta) \int_0^\tau \intTN{ 1_{\rm res} + \vr_{\rm res} + \vr |\vu - \vU|^2 } \dt
+ \delta 
 \intO{ \left( \mathbb{S}(\Grad \vu) - \mathbb{S}(\Grad \vU) \right): \left( \Grad \vu - \Grad \bf{U} \right) }, 
\end{align*} 
where we use again the Korn-Poincar\'e inequality \eqref{KoPo}.

Gathering all these estimates, we obtain the following inequality:
\begin{equation}\label{rel_ener1}
\begin{split}
&\left[ \mathcal{E}\left( \vr, \cc, \vu \ \Big| R, \CC, \vc{U} \right) \right]_{t = 0}^{ t = \tau} +  \int_0^\tau \intTN{ \Big[(1-2\delta) \mathbb{S}(\Grad \vu-\Grad \vc{U}) : \left( \Grad \vu 
- \Grad \bf{U} \right)  +  \mu^2 \Big] } \dt + \int_{\Ov{\Omega}} {\rm d} \EE(\tau)\\
& \leq - \int_0^\tau \intTN{ \mu(\Del \cc -\mu -2 \cc)} \dt
-  \int_0^\tau \intTN{ (\Div \vu -\Div \vc{U} )W(\cc)} \dt \\
&+\int_0^\tau \intTN{\dfrac{\vr}{R}(\vu -\vc{U}) \cdot \Grad \CC \Del \CC}\dt
+ \int_0^\tau \intTN{ (R-\vr)\partial_t P'(R) } \dt\\
&- \int_0^\tau \intTN{(\vc{U}- \vu) \cdot \Grad p(R)}\dt + \int_0^\tau \intTN{(\vc{U}- \vu) \cdot \Grad\left[\CC^2+W(\CC) \right]}\dt\\
&-\int_0^\tau \intTN{\vr \vu \cdot \Grad P'(R)}\dt
- \int_0^\tau \intTN{(p(\vr )-\cc^2) \Div \vc{U}} \dt\\
&- \int_0^\tau \intTN{ \left( \Grad c \otimes \Grad c - \frac{1}{2} |\Grad c|^2 \mathbb{I} \right) : \Grad {\bf{U}} } \dt
- \int_0^\tau \int_\Omega  \Grad \bf{U} : {\rm d}\RR(t)  \dt\\
&+ \int_0^\tau \intTN{ (\mu - \vu \cdot \Grad \cc) (\Del \CC - 2 \CC)} \dt 
+ \int_0^\tau \intTN{ (\cc-\CC) (\Del \CC_t- 2\CC_t)} \dt\\
&+ c(\delta, R, \CC, \vc{U})\int_0^\tau \mathcal{E}\left( \vr, \cc, \vu \ \Big| R, \CC, \vc{U} \right)(t) \dt.
\end{split}
\end{equation}

Now, notice that the continuity equation for $R$ and $\vc{U}$ also implies that
\begin{equation}\label{renR}
\partial_t P'(R) + \vc{U} \cdot \Grad P'(R) + R P''(R) \Div \vc{U}=0,
\end{equation}
with $R P''(R)=p'(R)$.

Using \eqref{renR}, we can handle the terms related to the elastic pressure $P$:
\begin{equation} 	\label{elpe}
\begin{split}
&- \int_0^\tau \intTN{ (\vc{U}- \vu) \cdot
\Grad p(R)  } \dt 
- \int_0^\tau \intTN{ p(\vr) \Div \vc{U} } \dt
\\
&- \int_0^\tau \intTN{ \vr \vc{u}  \cdot \Grad P'(R) } \dt + 
\int_0^\tau \intTN{(R-\vr) \partial_t P'(R) } \dt\\
&=\int_0^\tau \intTN{ \Div \vc{U} (p(R)-p(\vr)-(R-\vr)p'(R) ) } \dt + \int_0^\tau \intTN{ \vu \cdot \Grad p(R)}\dt\\
& - \int_0^\tau \intTN{(R-\vr) (\vc{U}-\vu) \cdot \Grad P'(R) }\dt  - \int_0^\tau \intTN{R \vu \cdot \Grad P'(R)}\dt.
\end{split}
\end{equation}

For the first term from the right hand side of \eqref{elpe} we have the following bound:
\begin{equation}
\big| \int_0^\tau \intTN{ \Div \vc{U} (p(R)-p(\vr)-(R-\vr)p'(R) ) } \dt\big| \aleq 
\int_0^\tau \mathcal{E}\left( \vr, \cc, \vu \ \Big| R, \CC, \vc{U} \right)(t) \dt,
\end{equation}
where we used the fact that:
\begin{equation}	\label{point-pF}
\big| p(R) - p'(R)(R-\vr) - p(\vr) \big|
\aleq
\big| P(\vr) - P'(R)(\vr - R) - P(R) \big|.
\end{equation}
The second and the last term from the right hand side of \eqref{elpe} cancel since $\Grad p(R)= R \Grad P'(R)$ and the third term is estimated as in \eqref{drR}:
\begin{equation}
    \begin{split}
    &\big|  \int_0^\tau \intTN{(R-\vr) (\vc{U}-\vu) \cdot \Grad P'(R) }\dt\big|\leq  \int_0^\tau \intTN{|R-\vr| |\vc{U}-\vu| } \dt\\
  & \leq  \delta \int_0^\tau \intO{ \left( \mathbb{S}(\Grad \vu) - \mathbb{S}(\Grad \vU) \right): \left( \Grad \vu 
- \Grad \bf{U} \right) }\dt
+ c(\delta, R, \vc{U}) \int_0^\tau \mathcal{E}(\vr, \vu, \cc| R, \vc{U}, \CC)(t) \dt.     
    \end{split}
\end{equation}
Returning to \eqref{rel_ener1}, we obtain:
\begin{equation}\label{rel_ener2}
\begin{split}
&\left[ \mathcal{E}\left( \vr, \cc, \vu \ \Big| R, \CC, \vc{U} \right) \right]_{t = 0}^{ t = \tau} +  \int_0^\tau \intTN{ \Big[(1-3\delta) \mathbb{S}(\Grad \vu-\Grad \vc{U}) : \left( \Grad \vu 
- \Grad \bf{U} \right)  +  \mu^2 \Big] } \dt + \int_{\Ov{\Omega}} {\rm d} \EE(\tau)\\
& \leq - \int_0^\tau \intTN{ \mu(\Del \cc -\mu -2 \cc)} \dt
-  \int_0^\tau \intTN{ (\Div \vu -\Div \vc{U} )W(\cc)} \dt \\
&+\int_0^\tau \intTN{\dfrac{\vr-R}{R}(\vu -\vc{U}) \cdot \Grad \CC \Del \CC}\dt 
+\int_0^\tau \intTN{(\vu -\vc{U}) \cdot \Grad \CC \Del \CC}\dt\\
 &+ \int_0^\tau \intTN{(\vc{U}- \vu) \cdot \Grad \CC^2 }\dt+ \int_0^\tau \intTN{(\vc{U}- \vu) \cdot \Grad W(\CC)}\dt\\
&+ \int_0^\tau \intTN{\cc^2 \Div \vc{U}} \dt
- \int_0^\tau \intTN{ \left( \Grad c \otimes \Grad c - \frac{1}{2} |\Grad c|^2 \mathbb{I} \right) : \Grad {\bf{U}} } \dt\\
&- \int_0^\tau \int_{\Omega}  \Grad {\bf{U}} : {\rm d}\RR(t)  \dt+ \int_0^\tau \intTN{ \mu   (\Del \CC - 2 \CC)} \dt - \int_0^\tau \intTN{  \vu \cdot \Grad \cc (\Del \CC - 2 \CC)} \dt\\
&+ \int_0^\tau \intTN{ ( \cc- \CC) (\Del \CC_t- 2\CC_t)} \dt+ c(\delta, R, \CC, \vc{U})\int_0^\tau \mathcal{E}\left( \vr, \cc, \vu \ \Big| R, \CC, \vc{U} \right)(t) \dt.
\end{split}
\end{equation}
We first remark that:
\begin{equation}
    \begin{split}
 &  \big|  \int_0^\tau \intTN{ (\Div \vu -\Div \vc{U} )W(\cc)} \dt   - \int_0^\tau \intTN{(\vc{U}- \vu) \cdot \Grad W(\CC)}\dt\big|\\  
   &= \big| \int_0^\tau \intTN{ (\Div \vu -\Div \vc{U} )(W(\cc) - W(\CC))} \dt  \big| \\
   &\aleq
    \int_0^\tau \intTN{ |\Div \vu -\Div \vc{U}| |\cc - \CC|} \dt \\
   & \leq \delta \int_0^\tau \intO{ \left( \mathbb{S}(\Grad \vu) - \mathbb{S}(\Grad \vU) \right): \left( \Grad \vu 
- \Grad \bf{U} \right) }\dt
+ c(\delta) \int_0^\tau \mathcal{E}(\vr, \vu, \cc| R, \vc{U}, \CC)(t) \dt.    
    \end{split}
\end{equation}
Since $R$ and $\CC$ are regular, we can also write:
\begin{equation}
    \begin{split}
      \big| \int_0^\tau \intTN{\dfrac{\vr-R}{R}(\vu -\vc{U}) \cdot \Grad \CC \Del \CC}\dt \big|   \aleq \int_0^\tau \intTN{|\vr-R||\vu -\vc{U}|}\dt,
    \end{split}
\end{equation}
and the right hand side we estimate identically as in \eqref{drR}.

Using the equation for the concentration $\CC$, we have:
\begin{equation}
    \begin{split}
      &\int_0^\tau \intTN{ ( \cc- \CC) (\Del \CC_t- 2\CC_t)} \dt\\
      &=\int_0^\tau \intTN{\left[\Del (\cc-\CC) - 2(\cc-\CC) \right] \left[ \Del \CC -2 \CC -W'(\CC)\right] } \dt\\
     & -\int_0^\tau \intTN{\Del (\cc-\CC) \vc{U} \cdot \Grad \CC} \dt+2\int_0^\tau \intTN{(\CC-\cc) \vc{U}\cdot \Grad \CC } \dt.
    \end{split}
\end{equation}
Gathering the following terms, we also obtain:
\begin{align*}
\int_0^\tau &\intTN{ (\vu - \vU) \cdot \Grad \CC \Del \CC } \dt  - 
\int_0^\tau \intTN{ \left( \Grad c \otimes \Grad c - \frac{1}{2} |\Grad c|^2 \mathbb{I} \right) : \Grad {\bf{U}} } \dt\\
&
- \int_0^\tau \intTN{ \vu \cdot \Grad \cc \Del \CC } \dt 
- \int_0^\tau \intTN{  (\Del \cc-\Del \CC)\vc{U} \cdot \Grad \CC } \dt 
\\
&= \int_0^\tau \intTN{ \vu \cdot \Grad \CC \Del \CC } \dt  + 
\int_0^\tau \intTN{ \vU \cdot \Grad \cc \Del \cc } \dt\\
&- \int_0^\tau \intTN{ \vc{U} \cdot \Grad \CC \Del \cc } \dt 
- \int_0^\tau \intTN{ \vu \cdot \Grad \cc \Del \CC } \dt \\
& = \int_0^\tau \intTN{ \Del \CC (\Grad \CC - \Grad \cc) \cdot (\vu - \vU) } \dt 
+ \int_0^\tau \intO{ \vU \cdot (\Grad \CC - \Grad \cc) (\Del \CC - \Del \cc ) } \dt\\ 
&= \int_0^\tau \intTN{ \Del \CC (\Grad \CC - \Grad \cc) \cdot (\vu - \vU) } \dt \\&- 
\int_0^\tau \intO{ \Grad \vU : \left[ \Grad (\CC - \cc) \otimes \Grad (\CC-\cc) - \frac{1}{2} |\Grad (\CC-\cc)|^2 \mathbb{I} \right] } \dt,
\end{align*}
terms that can be also bounded by:
$$
 \delta \int_0^\tau \intO{ \left( \mathbb{S}(\Grad \vu) - \mathbb{S}(\Grad \vU) \right): \left( \Grad \vu 
- \Grad \bf{U} \right) }\dt
+ c(\delta, \vc{U}, \CC) \int_0^\tau \mathcal{E}(\vr, \vu, \cc| R, \vc{U}, \CC)(t) \dt.   
$$
We also gather the following convective terms:
\begin{equation}
    \begin{split}
       & \int_0^\tau \intO{(\vc{U}-\vu)\cdot \Grad \CC^2 }\dt +
         \int_0^\tau \intO{ \cc^2 \Div \vc{U}}\dt\\
        & + 2  \int_0^\tau \intO{\vu \cdot \Grad \cc \CC }\dt
         -2\int_0^\tau \intO{(\cc-\CC) \vc{U} \cdot\Grad \CC }\dt\\
       &  =\int_0^\tau \intO{\Div \vc{U} (\CC-\cc)^2 }\dt + 
       2\int_0^\tau \intO{(\vu- \vc{U})\cdot (\Grad \cc- \Grad \CC) \CC }\dt\\
      & \leq \delta \int_0^\tau \intO{ \left( \mathbb{S}(\Grad \vu) - \mathbb{S}(\Grad \vU) \right): \left( \Grad \vu 
- \Grad \bf{U} \right) }\dt
+ c(\delta, \vc{U}, \CC) \int_0^\tau \mathcal{E}(\vr, \vu, \cc| R, \vc{U}, \CC)(t) \dt.   
    \end{split}
\end{equation}
The remaining terms give, after elementary manipulations:
\begin{equation}\label{muM}
    \begin{split}
      &  -\int_0^\tau \intO{\mu(\Del \cc - \mu -2\cc) }\dt + \int_0^\tau \intO{ \mu (\Del \CC-2\CC) }\dt\\
      &  + \int_0^\tau \intO{\left[\Del(\cc-\CC)-2(\cc-\CC) \right] \left[\Del \CC -2\CC -W'(\CC) \right] }\dt\\
     &= \int_0^\tau \intO{\mu^2}\dt - \int_0^\tau \intO{\left[ \Del (\cc-\CC)-2(\cc-\CC)\right] \left[ \mu - \Del \CC -2 \CC -W'(\CC)\right] }\dt\\
     &= \int_0^\tau \intO{\mu^2}\dt - \int_0^\tau \intO{| \Del (\cc-\CC)-2(\cc-\CC)|^2  }\dt\\
     &- \int_0^\tau \intO{\left[ \Del (\cc-\CC)-2(\cc-\CC)\right] \left[ W'(\cc) -W'(\CC)\right] }\dt,
    \end{split}
\end{equation}
where we can bound the last term in \eqref{muM} by:
\begin{equation}\begin{split}
  &  \Big| \int_0^\tau \intO{\left[ \Del (\cc-\CC)-2(\cc-\CC)\right] \left[ W'(\cc) -W'(\CC)\right] }\dt\big| \\
   & \leq \dfrac12 \int_0^\tau \intO{|\Del (\cc-\CC)-2(\cc-\CC)|^2 }\dt + \dfrac12 \int_0^\tau \intO{(\cc-\CC)^2 } \dt.
\end{split}\end{equation}
Gathering all these estimates in \eqref{rel_ener2} and taking $\delta$ small enough, we obtain:
\begin{equation}\label{rel_ener3}
\begin{split}
&\left[ \mathcal{E}\left( \vr,  \vu, \cc \ \Big| R,  \vc{U} , \CC\right) \right]_{t = 0}^{ t = \tau} + \dfrac12 \int_0^\tau \intTN{ \mathbb{S}(\Grad \vu-\Grad \vc{U}) : \left( \Grad \vu 
- \Grad \bf{U} \right) }\dt \\
&+ \int_{\Ov{\Omega}} {\rm d} \EE(\tau) + \dfrac12 \int_0^\tau \intO{|\Del (\cc-\CC)-2(\cc-\CC)|^2 }\dt\\
& \leq 
- \int_0^\tau \int_\Omega  \Grad {\bf{U}} : {\rm d}\RR(t) \dt+ c(\delta, R, \CC, \vc{U})\int_0^\tau \mathcal{E}\left( \vr,  \vu, \cc \ \Big| R, \vc{U}, \CC \right)(t) \dt.
\end{split}
\end{equation}
Using the defect compatibility hypothesis \eqref{D6} (see also Remark~\ref{DR1}), we also have:
\begin{equation}\label{defect}
    \big| \int_0^\tau \intTN{  \Grad {\bf{U}} : {\rm d}\RR(t) } \dt\big| \leq  \Lambda \| \Grad {\bf{U}} \|_{L^\infty((0,T) \times \Omega; \R^{N\times N})}  \int_0^\tau \int_\Omega{{\rm d}\EE}(t) \dt.
\end{equation}
Using \eqref{defect} with \eqref{rel_ener3}, the relative energy inequality reduces to:
\begin{equation}\begin{split}
     \mathcal{E}\left( \vr,  \vu, \cc \ \Big| R,  \vc{U}, \CC \right)(\tau) &+  \int_{\Ov{\Omega}}  {\rm d} \EE(\tau) \leq  \mathcal{E}\left( \vr_0, \vu_0, \cc_0 \ \Big| R(0, \cdot), \vc{U}(0, \cdot), \CC(0, \cdot) \right)\\
   &+ c(\delta, R,  \vc{U}, \CC)\int_0^\tau \left[\mathcal{E}\left( \vr,  \vu, \cc \ \Big| R,  \vc{U}, \CC \right)(t)  + \int_{\Ov{\Omega}} {\rm d} \EE(t) \right]\dt.
\end{split}\end{equation}
Applying the Gronwall lemma, we finally obtain the desired conclusion.
\end{proof}

\section{Convergence of a numerical approximation}
In this section we propose a combined discontinuous Galerkin (DG) -- finite element (FE) method for the approximation of the Navier--Stokes--Allen--Cahn system~\eqref{i1}. Specifically,  for the Navier--Stokes part, we adopt the method studied by Karper~\cite{Karper} as well as Feireisl and Luka{\v c}ov{\' a}~\cite{FeiLuk}. For the Allen-Cahn part, we take a discontinuous Galerkin approximation. 
The main purpose is to analyze the convergence of DG-FE  method using the theoretical study built in the previous sections. 
For the sake of simplicity, we restrict ourselves to the space periodic boundary conditions, meaning $\Omega=\mathcal{T}^d$.

Moreover, we strengthen the hypothesis \eqref{H3} concerning the structural properties of the pressure. Here and hereafter we suppose that
\begin{equation} \label{H3bis}
\begin{split}
&p \in C[0, \infty) \cap C^2(0, \infty),\ 
p(0) = 0, \ p'(\vr) > 0 \ \mbox{for}\ \vr > 0; \\
&\mbox{the pressure potential}\ P 
\ \mbox{determined by}\ P'(\vr) \vr - P(\vr) = p(\vr)\ \mbox{satisfies}\
P(0) = 0,\\ &\mbox{and}\ P - \underline{a} p,\ \Ov{a} p - P\ \mbox{are convex functions 
for certain constants}\ \underline{a} > 0, \ \Ov{a} > 0.
\end{split}
\end{equation} 
As shown in \cite[Section 2.1.1]{AbbFeiNov}, hypothesis \eqref{H3bis} implies that there exists $\gamma > 1$ such that 
\begin{equation} \label{H3a}
P(\vr) \geq a \vr^\gamma \ \mbox{for some}\ a > 0 \ \mbox{and all}\ \vr \geq 1.
\end{equation}

\subsection{Notations} \label{Sec_nt}
We begin by introducing the  notations. We write  $A \aleq B$ if $A \leq cB$ for a generic positive constant $c$ independent of discretization parameters $\TS$ and $h$.  
We denote the norms $\norm{\cdot}_{L^q(\Omega)}$ and $\norm{\cdot}_{L^p(0,T;L^q(\Omega))}$ by $\norm{\cdot}_{L^q}$ and $\norm{\cdot}_{L^pL^q}$, respectively.  Moreover, we denote $\co{A}{B}=[\min\{A,B\}, \max\{A,B\}]$.

\paragraph{Mesh.} Let $\grid = \grid_h$ be a regular and quasi-uniform triangulation of $\Omega\equiv \left([-1,1]|_{\{-1,1\}} \right)^d$ in the sense of Ciarlet~\cite{ciarlet}, where $h$ is the mesh size defined below. %
 Moreover, let $\grid$ be periodic in the sense 
 of Definition~\ref{rmq_pbc}.  
We use the following notations:
\begin{itemize}
\item We denote $K$ a generic element such that $\Omega = \cup_{K\in \grid} K$. For any element $K$ we denote  $|K|$ its volume and $h_K$ its diameter. Further, we define $h=\max_{K\in \grid} h_K$ as the size the mesh. 
\item We denote by $\faces$ the set of all faces,   $\facesK$ the set of faces of an element $K \in \grid$. 
By $|\sigma|$ we denote the volume of the face $\sigma \in \faces$. 
Note that each $\sigma \in \faces$ is an interior edge due to the periodicity assumption, i.e., there exist two different elements $K\in \grid$ and $L\in\grid$ such that $\sigma=\facesK \cap \faces(L)$ for all $\sigma \in \faces$, which we often note $\sigma=K|L$. 
\item  
For each face $\sigma \in \faces$, we denote by $\vn$ its outer normal vector. If furthermore $\sigma \in \facesK$ (resp. $\faces(L)$) we write it as $\vn_K$ (resp. $\vn_L$).
\end{itemize}

Periodic boundary conditions frequently appear in the mathematical physical problems. Their numerical realization is often more complicated than that of Dirichlet or Neumann type boundary conditions. We realize the periodicity by the following definition.  
\begin{Definition}[Periodic mesh]\label{rmq_pbc} 
 Let $\grid$ be a triangulation of $\Omega \subset \R^d$. 
 Let $P^L_i$ (resp. $P^R_i$), $i=1,\dots, d$, be the set of vertices that forms the edges on the left (resp. right) boundary of $\Omega$ in the $i^{\rm th}$ direction of the Cartesian coordinates. We say $\grid$ is periodic mesh if the following conditions are satisfied:
 \begin{enumerate}
 \item
For any vertex $P \in P^L_i$, there exists a dual vertex $P^* \in P^R_i$ such that $x_{P^*} -x_{P} =\ell_i \bfe_i$, where $\bfe_i$ is the $i^{\rm th}$ basis vector of the Cartesian coordinates and $\ell_i$ is the length of the domain $\Omega$ in the $i^{\rm th}-$ direction. 
\item For all $i=1,\dots,d$, the vertices $P \in P^L_i$ and their dual $P^* \in P^R_i$ are treated as the same degree of freedom. 
\end{enumerate}
\end{Definition}
\medskip

For a piecewise (elementwise) continuous function $v$ we define
\begin{align*}
v^{\rm out}(x) = \lim_{\delta \to 0+} v(x + \delta \vn),\quad
v^{\rm in}(x) = \lim_{\delta \to 0+} v(x - \delta \vn),\\
\avs{v}(x) = \frac{v^{\rm in}(x) + v^{\rm out}(x) }{2},\quad
\jump{ v }  (x)= v^{\rm out}(x) - v^{\rm in}(x)
\end{align*}
whenever $x \in \sigma \in \facesint$.  { Note that our jump operator has an opposite sign with respect to the classical discontinuous Galerkin setting \cite{DPE}. }

\noindent
\paragraph{Function spaces.}
Let ${\cal P}^{\ell}_d(K)$ be the space of polynomials of degree not greater than $\ell$ on $K$ for $d-$dimensional vector-valued functions. We introduce the following function spaces: 
\begin{align*}
\Qh & =\left\{ v\in L^1(\Omega)| v_K  \in \calP^0_1(K) \; \forall K\in \grid \right\},
\\
\bVh &=\left\{ \vv\in L^2(\Omega)| \vv_K \in \calP^1_d(K)\; \forall K\in \grid ; \intSh{\jump{\vv}} =0 \; \forall \sigma \in \faces \right\},
\\
\Xh &=  \left\{  v\in L^2(\Omega)| v_K  \in \calP^1_1(K) \; \forall K\in \grid \right\},
\end{align*}
 associated with the following projection operators 
\begin{equation*}
\Piq: \, L^1(\Omega) \rightarrow \Qh, \qquad
\Piv: \, W^{1,2}(\Omega) \rightarrow \bVh,  \qquad
\Pix: \, W^{2,2}(\Omega) \rightarrow \Xh. 
\end{equation*}

Moreover, we introduce the space
\[ 
\Wh:=  \left\{  v\in \Xh \ \middle| \ \intO{v}=0 \right\}
\]
along with the projection operator $\Piw$ constructed by Kay et al.~\cite{Kay},
enjoying the following properties, see \cite[Section 2, formula (2.20)]{Kay} 
\begin{equation} \label{crucial}
\| \Piw v - v \|_{L^2(\Omega)} \leq h \normH{ \Piw v - v } \quad \mbox{and}\quad
\normH{ \Piw v - v }   \aleq {h^{1 - \beta} \| v \|_{W^{2,2}(\Omega)}} 
\end{equation}
for any $v \in W^{2,2}(\Omega).$   Here we have introduced the broken norm
\[
\normH{v}^2 = \sum_{K\in \grid} \intK{|\Gradh v|^2} + h \intfaces{ \avs{\Gradh v}^2}  +
{\frac{1}{h^{1 + \beta}} \intfaces{\jump{v}^2}}, \quad \mbox{with } \beta>0.
\]

Note that the operator $\Piq$ can be explicitly written as  
\[\Piq  \phi  = \sum_{K \in \grid}  \frac{1_{K}(x) }{|K|} \int_K \phi \dx, \quad 1_K= \begin{cases}
 1 & \text{if }x\in K, \\
 0 & \text{otherwise.}
\end{cases}\]
We shall frequently use the notation $\widehat{\phi} =\Piq \phi$. 
Hereafter, for any $K \in \grid$ we denote: 
\[\nabla_h v|_K = \Grad  v|_K, \quad 
\Divh \vu|_K = \Div \vu|_K
\]
for any $v\in \bVh \cup \Xh$, $\vu \in \bVh$.

Further, we introduce the bilinear form 
\begin{equation*}
B(v,w) = \intO{\Gradh v \cdot \Gradh w} 
{ + }  \intfacesB{\jump{w} \vn \cdot \avs{\Gradh v}  {+}  \jump{v} \vn \cdot \avs{\Gradh w}  {+ }  
{\frac{1}{h^{1 + \beta}}} \jump{v}\jump{w}}.
\end{equation*}

It is easy to check 
\begin{equation}\label{dtB}
    B(v, v - w) 
 = \frac12  \normB{v}^2   -  \frac12 \normB{w}^2  + \frac12   \normB{v - w}^2  .
\end{equation}

Next, we introduce a norm on $\Wh$ (seminorm on $\Xh$), 
\begin{align}\label{seminorm}
    \normB{v}^2 = { \sum_{K\in \grid} \intK{|\Gradh v|^2} +\frac{1}{h^{1 + \beta}} }\intfaces{\jump{v}^2}.
\end{align}
Observe that the seminorms $\normB{\cdot}$ and $\normH{ \cdot }$ are equivalent on  $\Xh$ with a constant independent of $h$.
Consequently, by means of the Riesz representation theorem, there exists a unique $\Laph v \in \Wh$ such  that 
\begin{equation}\label{bilinear}
-  \intO{\Laph v\; w} = B(v, w) \quad \mbox{for any } w \in \Wh.
\end{equation}
Here we may replace the test function space $\Wh$ by $\Xh$ as $\Xh=\Wh \oplus {\rm span} \{1\}$ and $w\equiv 1$ satisfies \eqref{bilinear}. 

\begin{Lemma}[Closed graph lemma] \label{cL1}
Suppose that $v_h(t) \in \Xh$ for a.a. $t \in (0,T)$,
\[
\sup_{t \in (0,T) } \normB{v_h} \aleq 1,
\]
and
\[
\ v_h \to v \ \mbox{weakly in}\ L^2(0,T;L^2(\Omega)),\ 
\Laph v_h \to \widetilde{\Lap  v} \ \mbox{weakly in}\ L^2(0,T;L^2(\Omega)).
\]

Then 
\[
\widetilde{\Lap  v} = \Lap  v \ \mbox{in}\ \mathcal{D}'((0,T) \times \Omega).
\]

\end{Lemma}

\begin{proof}
Our goal is to show 
\[
- \int_0^T \intO{ \widetilde{ \Lap  v} w } \dt = 
-  \int_0^T \lim_{h \to 0} \intO{ \Laph v_h w } \dt = - \int_0^T \intO{ v \Lap  w } \dt
\]
for any $w \in L^2(0,T; W^{2,2}(\Omega))$. Without loss of generality, we may assume $\intO{ w } = 0$ for a.a. $t$.

We have 
\[
- \int_0^T \intO{ \Laph v_h w } \dt = - \int_0^T \intO{ \Laph v_h (w - \Piw w) } \dt - \int_0^T \intO{ \Laph v_h \Piw w } \dt,   
\]
where, by virtue of \eqref{crucial}, 
\[
\int_0^T \intO{ \Laph v_h (w - \Piw w) } \dt \to 0 \ \mbox{as}\ h 
{   \to 0}.
\]

Next, in accordance with \eqref{bilinear},
\[ 
- \intO{ \Laph v_h \Piw w } = B (v_h, \Piw w) = B(v_h, \Piw w - w) + B(v_h, w), 
\]
where, by direct manipulation,

\[
B(v_h, w) = - \intO{ v_h \Lap  w }\ \mbox{and, in particular,}\ \int_0^T B(v_h, w) \dt \to - \int_0^T \intO{ v \Lap  w } 
\ \mbox{as}\ h \to 0.
\]
Thus it is enough to show 
\[
\int_0^T B(v_h, \Piw w - w) \dt \to 0 \ \mbox{as}\ h \to 0.
\]
It follows from Cauchy--Schwartz inequality that 
\[
|B(v_h, \Piw w - w)| \aleq \|| v_h \|| \ \|| \Piw w - w \||.
\]
As $v_h \in \Xh$, we have 
\[
\|| v_h \|| \aleq \normB{v_h} 
\]
and the desired conclusion follows from \eqref{crucial}.

\end{proof}

\begin{Lemma} [Compactness Lemma] \label{cL2}

Suppose that $v_h(t) \in \Xh$ for a.a. $t \in (0,T)$, 
\[
v_h(t) \to v \ \mbox{(strongly) in}\ L^2(0,T; L^2(\Omega)),\ 
\Laph v_h(t) \to \Lap  v \ \mbox{weakly in}\ L^2(0,T; L^2(\Omega)), 
\]
and
\[
\nabla_h v_h \to \Grad v \ \mbox{weakly in}\ L^2(0,T; L^2(\Omega; \R^d)).
\]

Then
\[
\nabla_h v_h \to \Grad v \ \mbox{(strongly) in}\ L^2(0,T; L^2(\Omega; \R^d)).
\]

\end{Lemma}

\begin{proof}
In view of the weak lower semi--continuity and convexity of the $L^2$-norm, it is enough to show 
\[
\limsup_{h \to 0} \int_0^T \intO{ |\nabla_h v_h |^2 } \dt \leq \int_0^T \intO{ |\Grad v|^2 } \dt. 
\]
To see this, we {write 
\[
\intO{ |\nabla_h v_h |^2 }  = - \intO{ \Laph  v_h v_h } 
- \intfacesB{2 \jump{v_h} \vn \cdot \avs{\Gradh v_h}   + 
\frac{1}{h^{1 + \beta}} \jump{v_h}^2}.\]
}
On one hand, thanks to our hypotheses, 
\[
- \int_0^T \intO{ \Laph  v_h v_h } \dt \to - \int_0^T \intO{ \Lap  v v } \dt =  \int_0^T \intO{ |\Grad v|^2 } \dt. 
\]
{On the other hand,
\[
\left| \intfaces{2 \jump{v_h} \vn \cdot \avs{\Gradh v_h} } \right| \leq \intfaces{\frac{1}{h^{1 + \beta}} \jump{v_h}^2 } 
+ c h^\beta \intO{ |\nabla_h v_h |^2 }. 
\]
As $\beta > 0$, we get the desired conclusion.}
\end{proof}

\paragraph{Diffusive upwind flux.}
Given the velocity field $\vv \in \bVh$, the upwind flux for any function $r\in \Qh$ is specified at each face  $\sigma \in \faces$ by
\begin{align*}\label{Up}
\Up [r, \vv]|_\sigma   =r^{\up} \vs \cdot \vn
=r^{\rm in} [ \vs \cdot \vn]^+ + r^{\rm out} [ \vs \cdot \vn]^-
= \avs{r} \ \vs \cdot \vn - \frac{1}{2} |\vs \cdot \vn| \jump{r},
\end{align*}
where
\begin{equation*}
\vs = \frac{1}{|\sigma|} \intSh{\vv}, \quad 
[f]^{\pm} = \frac{f \pm |f| }{2} \quad \mbox{and} \quad
r^{\up} =
\begin{cases}
 r^{\rm in} & \mbox{if} \ \vs \cdot \vn \geq 0, \\
r^{\rm out} & \mbox{if} \ \vs \cdot \vn < 0.
\end{cases}
\end{equation*}
Furthermore, we consider a diffusive numerical flux function of the following form
\begin{align}\label{flux_F}
\Fup(r,\vv)
=\Up[r, \vv] - h^{\eps} \jump{ r }, \, \eps>0. 
\end{align}

When ${\bf r}$ is a vector function, e.g. ${\bf r}= \vr \vu$ in the momentum equation, we write the above numerical flux as 
\[ \bFup(\vr \vu, \vv) \equiv \big( \Fup(\vr u_1, \vv), \ldots, \Fup (\vr u_d, \vv)  \big)^T \mbox{ and } {\bUp}(\vr \vu, \vv) \equiv \big( \Up(\vr u_1, \vv), \ldots, \Up (\vr u_d, \vv)  \big)^T.\]


\paragraph{Time discretization.}
For a given time step $\TS \approx h >0$,
we denote the approximation of a function $v_h$ at time $t^k= k\TS$ by $v_h^k$ for $k=1,\ldots,N_T(=T/\TS)$.  
Then, we introduce the piecewise constant extension of discrete values,
\begin{equation}\label{TD}
 v_h(t) =\sum_{k=1}^{N_T} v_h^k 1_{I^k} 
\mbox{ with } I^k= ((k-1)\TS,k\TS].
\end{equation}
Furthermore, we approximate the time derivative by the backward Euler method
\[D_t v_h(t)
 =\frac{v_h(t) - v_h(t - \TS )}{\TS} \; \forall \; t \in(0,T], i.e., \quad 
 D_t v_h^k = \frac{v_h^k - v_h^{k-1} }{\TS} \mbox{ for all } k=1,\dots, N_T.
\]

\paragraph{Useful estimates.}
We recall some basic inequalities used in the numerical analysis. First, thanks to Taylor's theorem, it is obvious for $\phi, \bfphi \in C^2(\Omega)$ that  
\begin{equation}\label{n4c2}\begin{split}
 \norm{ \phi -  \Piq  \phi  }_{L^p} \aleq  h^s \norm{ \phi }_{C^s}, \>
\norm{  \bfphi -   \Piv \bfphi   }_{L^p}  \aleq h^s\norm{ \phi }_{C^s}, \\
 \norm{ \phi -  \Pix  \phi  }_{L^p} \aleq  h^s \norm{ \phi }_{C^s}, \ 1< p \leq \infty , \>
  \mbox{ for } s=1,2.
\end{split}\end{equation}
 Next, we report a discrete analogous of the Poincar\' e--Sobolev type inequality (see \cite[Theorem 17]{FeLMMiSh} for a similar result):
\begin{Lemma}[Sobolev inequality]\label{lm_SP}
 Let  $r \geq 0$ be a function defined on $\Omega \subset \R^d$  such that
\[ 0< c_M \leq \intO{r} , \mbox{ and } \intO{r^\gamma}\leq  c_E \mbox{ for } \gamma>1,
\]
where $c_M$ and $c_E$ are some positive constants.
Then the following Poincar{\'e}-Sobolev type inequality holds true
\begin{equation}\label{SP}
\norm{ v_h }_{L^q(\Omega)}^2 
\aleq  c \norm{\Gradh v_h }_{L^2(\Omega)}^2  
+c \intO{r|\Pi_a v_h |^2}
\end{equation}
for any $v_h \in \bVh\cup \Xh$, and $1 \leq q \leq 6$ for $d=3$, $1\leq q<\infty$ for $d=2$,
where the constant $c$ depends on $c_M$ and $c_E$ but not on the mesh parameter and $\Pi_a \in \{1, \Piq \}$.  
In particular, setting $r=1$ yields
\begin{equation}\label{SP2}
    \norm{ v_h }_{L^q(\Omega)}^2 
\aleq  \norm{\Gradh v_h }_{L^2(\Omega)}^2  
+  \norm{ v_h }_{L^2(\Omega)}^2 .
\end{equation}
\end{Lemma}

The following lemma shall be useful in the analysis of the energy stability. 
\begin{Lemma}\label{lm_flux}
For any $\vrh \in \Qh$ and $\vuh \in \bVh$ it holds
\begin{equation}\label{Flux}\begin{split}
&\intfacesB{ \bFup(\vrh \auh, \vuh) \cdot \jump{\auh}  -   \bFup(\vrh, \vuh) \jump{\frac12|\auh|^2} }
\\ & = -\frac12 \intfaces{ \vrh^{\up} \abs{ \jump{\auh}}^2  \abs{\us \cdot \vn}}
 - h^\eps \intfaces{  \avs{\vrh} \abs{\jump{\auh}}^2}, 
\end{split}\end{equation}
where we have denoted $\us=\frac{1}{|\sigma|}\intSh{\vuh}$.
\end{Lemma}
The proof is analogous to \cite[Lemma 8.1]{FeLMMiSh}. For completeness, we attach the proof in Appendix~\ref{Sec_Ap_flux}.

\subsection{A mixed discontinuous Galerkin -- finite element method}\label{secSKM}
Now we are ready to introduce a combined discontinuous Galerkin (DG) -- finite element (FE) method for the approximation of Navier--Stokes--Allen--Cahn system~\eqref{i1}--\eqref{i2} with the periodic boundary conditions \eqref{i3}. 
\begin{Definition}[DG-FE method]
Let $(\vrh^0,\vuh^0, \cch^0) =(\Piq\vr_0, \Piv \vu_0, \Pix \cc_0)$ be the initial data. We say $(\vrh, \vuh,\cch) = \sum\limits_{k=1}^{N_T}(\vrh^k,\vuh^k, \cch^k)1_{I^k}$ is a DG-FE approximation of the Navier--Stokes--Allen--Cahn system~\eqref{i1}--\eqref{i3} if the triple $(\vrh^k,\vuh^k, \cch^k) \in \Qh \times \bVh \times \Xh$ satisfies
the following system of algebraic equations with the space periodic boundary conditions for all $k=1,\dots, N_T$: 
\begin{subequations}\label{scheme}
\begin{equation}\label{schemeD}
\intO{ D_t \vrh^k \phi_h } -  \intfaces{  \Fup(\vrh^k,\vuh^k) \jump{\phi_h}   } = 0, \quad \mbox{for all } \phi_h \in \Qh; 
\end{equation}
\begin{equation}
\begin{split}\label{schemeM}
&\intO{ D_t  (\vrh^k \auhk) \cdot \bfphi_h } -  \intfaces{ \bFup(\vrh^k \auhk,\vuh^k)
\cdot \jump{\widehat{ \bfphi_h} }   }
  +  \nu \intO{ \Gradh \vuh^k: \Gradh \bfphi_h }\\&
+   \eta \intO{\Divh \vuh^k \Divh \bfphi_h}
=  \intO{ p_h^k  \Divh \bfphi_h  } 
+ \intO{ (f_h^k - \Laph \cch^k) \Gradh \cch^{k} \cdot  \bfphi_h}
 ,
\quad \mbox{for all }
\bfphi_h \in \bVh; 
\end{split}
\end{equation}
\begin{equation}\label{schemeC}
 \intO{ (D_t \cch^k  + \vuh^k \cdot \Gradh \cch^k) \psi_h }=  \intO{ \big(\Laph \cch^k -  \Fc_h^k  \big)   \psi_h   },  \quad \mbox{for all } \psi_h \in \Xh;
\end{equation}
where $\Laph \cch^k  \in \Wh$ is defined according to \eqref{bilinear}, 
\begin{equation}\label{schemeLap}
- \intO{\Laph \cch^k\; \varphi_h} = B(\cch^k,\varphi_h) \quad \mbox{for all } \varphi_h \in \Wh.
\end{equation}
\end{subequations}
Here, $p^k_h=p(\rho^k_h)$, $ \eta=\frac{d-2}{d}\nu + \lambda >0$ and
\begin{equation}\label{f_split}
\Fc_h^k =
\begin{cases}
 2(\cch^k + 1)  & \text{if } \cch^k \in(-\infty, -1) ,\\
(\cch^k)^3 - \cch^{k-1} & \text{if } \cch^k \in [-1,1]. \\
2( \cch^k - 1)  & \text{if } \cch^k \in (1, \infty).
 \end{cases}
\end{equation}
\end{Definition}

\begin{Remark}
Note that $\Fc_h^k$ is an approximation of $\Fc \equiv F'(\cc)$ at time $t^k$. 
The idea in defining $\Fc_h$ is that we split the convex and concave parts of $F(\cch)$ and respectively approximate them implicitly and explicitly in time. Such kind of splitting shall be helpful in deriving the energy stability, see the proof of Theorem~\ref{thm_sta} below. 
\end{Remark}

\subsection{Stability}\label{sec_St}

 In this section we show the stability of the DG-FE method, including the positivity of density, conservation of mass  and  the energy stability. 

\subsubsection{Basic properties of the DG-FE scheme}
Before stating the stability, let us show some fundamental properties of the DG-FE scheme \eqref{scheme}. 
\paragraph{Conservation of mass.}
Setting  $\phi_h =1$ in the density method \eqref{schemeD} we get 
$\intO{D_t \vrh^k} =0$, 
which implies for all $k=1,2,\ldots, N_T$ that 
\[ \intO{\vrh^k} = \intO{\vrh^{k-1}} =  \dots =\intO{\vrh^0} = \intO{\Piq \vr_0}= \intO{\vr_0}. 
\]
\paragraph{Internal energy balance.}
We recall the  discrete internal energy balance from~\cite[Section 4.1]{FeKaNo} or \cite[Lemma 3.1]{HS_MAC}. Indeed, testing the  scheme for the density \eqref{schemeD} by $\phi_h=\Hc'(\vrh^k)$ gives rise to the following lemma.
\begin{Lemma}[Discrete internal energy balance]\label{lem_r1}
Let $(\vrh^k,\vuh^k)\in  \Qh \times \bVh$ satisfy the discrete continuity equation \eqref{schemeD} for any $k\in\{1,\dots,N_T\}$. Then, there exist $\xi \in \co{\vrh^{k-1}}{\vrh^k}$ and  $\zeta \in \co{\vr_K^k}{\vr_L^k}$ for any $\sigma=K|L \in \facesint$  such that 
\begin{equation}\label{r1}
\begin{split}
& \intO{ D_t \Hc(\vrh^k)  }
+  \intO{ p(\vrh^k) \Divh \vuh^k    }
\\ &=
- \frac{\TS}2 \intO{ \Hc''(\xi)|D_t \vrh^k|^2  }
-  \intfaces{ \Hc''(\zeta) \jump{  \vrh^k } ^2 \left( h^\eps  + \frac12 |\us^k \cdot \vn | \right)} 
\leq 0.
\end{split}
\end{equation}
\end{Lemma}

\begin{Lemma}[Existence of a solution and positivity of density]\label{lem_exist}
Given $ \vr_0>0$.  For every $k=1, \dots, N_T$ there exists a solution $(\vrh^k,\vuh^k, \cch^k)\in \Qh \times \bVh \times \Xh$ to the DG-FE scheme \eqref{scheme}. 
Moreover, any solution to \eqref{scheme} preserves the positivity of the density, i.e.  $\vrh^k > 0$ for any $k=1, \dots, N_T$.
\end{Lemma}
The proof can be done analogously as \cite[Lemma 11.3]{FeLMMiSh}. For completeness, we give the proof in Appendix~\ref{Sec_Ap_exi}.

 \subsubsection{ Energy estimates}
Now, we are ready to  derive the discrete counterpart of the total energy balance \eqref{i5}.
\begin{Theorem}[Discrete energy  balance]\label{thm_sta}
Let $(\vrh,\vuh, \cch)$ be a solution of the DG-FE method \eqref{scheme}. Then we have the following energy estimate
\begin{equation} \label{sta}
\begin{split}
& D_t \intOB{ \frac{1}{2} \vrh^k  \abs{\auhk}^2  + P(\vrh^k)}
+  D_t \left(\intO{  F(\cch^k)} + \frac12 \normB{\cch^k}^2  \right)
\\&+  \nu   \norm{\Gradh \vuh^k}_{L^2}^2  + \eta \norm{  \Divh \vuh^k}_{L^2}^2
+ \norm{ D_t \cch^k + \vuh^k \cdot \Gradh \cch^k } _{L^2}^2  
= -  D_{\rm num}^k ,
\end{split}
\end{equation}
where 
$D_{\rm num}^k \geq 0$ is the  numerical dissipation 
\begin{equation} \label{dnum}
\begin{split}
 D_{\rm num}^k
&=   \frac{\TS}{2} \intO{ \vrh^{k-1}|D_t \auhk|^2  }
 + \frac12  \intfaces{ (\vrh^k)^{\rm up} \abs{\us^k \cdot \vn }  \abs{\jump{ \auhk }}^2 }
+ h^\eps \intfaces{  \avs{\vrh^k}  \abs{\jump{\auhk}}^2}
\\ & \quad 
+ \frac{\TS}2 \intO{ P''(\xi)|D_t \vrh^k|^2  }
+   \intfaces{ P''(\zeta) \jump{  \vrh^k } ^2 \left(h^\eps  + \frac12 | \us^k \cdot \vn | \right)}
\\&\quad  
+  \frac{\TS}{ 2 }     \normB{D_t \cch^k}^2 
 + \intO{  \frac{\TS}2 \big(1+ 1_{|\cch^k|> 1} +  3 (\cch^{k,*})^2 1_{|\cch^k|\leq 1} \big)  \abs{D_t \cch^k}^2}
, 
\end{split}
\end{equation}
where $\zeta \in \co{\vr_K^k}{\vr_L^k}$ for any $\sigma=K|L \in \facesint$,  $\xi \in \co{\vrh^{k-1}}{\vrh^k}$ and $\cch^{k,*} \in \co{\cch^{k-1}}{\cch^k}$. 
\end{Theorem}
\begin{proof}
First, setting $\bfphi_h = \vuh^k \in \bVh $ in \eqref{schemeM} we get
\begin{equation}\label{es1}
\begin{split}
& \intO{D_t (\vrh^k \auhk) \cdot \vuh^k  }
+ \nu   \norm{\Gradh \vuh^k}_{L^2}  + \eta \norm{  \Divh \vuh^k}_{L^2} 
\\& =
  \intfaces{ \bFup(\vrh^k \auhk,\vuh^k) \cdot \jump{\auhk  }   } 
  + \intO{ p_h^k  \Divh \vuh^k } 
+\intO{ ( \Fc_h^k - \Laph \cch^k )\Gradh \cch^k  \cdot \vuh^k}.
\end{split}
\end{equation}
Next, letting  $\phi_h = \frac{1}{2} \abs{\auhk}^2 \in \Qh$ in \eqref{schemeD} we find
\begin{equation}\label{es2}
\intO{ D_t \vrh^k   \frac{1}{2} \abs{\auhk}^2  } =  \intfaces{ \Fup [\vrh^k, \vuh^k ]  \jump{  \frac{1}{2} \abs{\auhk}^2 }  }.
\end{equation}
Subtracting \eqref{es2} from \eqref{es1} we derive 
\begin{equation} \label{es3}
\begin{split}
& D_t \intO{ \frac{1}{2} \vrh^k \abs{\auhk}^2  }
+  \nu   \norm{\Gradh \vuh^k}_{L^2}^2  + \eta \norm{\Divh \vuh^k}_{L^2} ^2
\\& =  - \frac{\TS}{2} \intO{ \vrh^{k-1} \abs{ D_t \auhk}^2  }
-\frac12 \intfaces{ \vrh^{k, \up} \abs{\jump{\auhk}}^2  \abs{\vu^k_\sigma \cdot \vn}}
\\&\quad   - h^\eps \intfaces{  \avs{\vrh^k} \abs{\jump{\auhk}}^2}
+  \intO{ p_h^k  \Divh \vuh^k  } 
+ \intO{ (f_h^k - \Laph \cch^k) \Gradh \cch^k \cdot  \vuh^k}
.\end{split}
\end{equation}
where we have used \eqref{Flux} and the following  identity
\[
 \intOB{ D_t (\vrh \auh)^k  \cdot \vuh^k - D_t \vrh^k \frac{\abs{\auhk}^2}{2} }
= \intOB{ D_t \Big(\frac12 \vrh^k \abs{\auhk}^2 \Big)   + \frac{\TS}{2} \vrh^{k-1} \abs{D_t\auhk}^2}. 
 \]
Further, by setting $ \psi_h = D_t \cch^k + \vuh^k \cdot \Gradh \cch^k \in \Xh$  in \eqref{schemeC}  we get 
\begin{equation}\label{es4}
\begin{aligned}
&\intO{\abs{ D_t \cch^k + \vuh^k \cdot \Gradh \cch^k  }^2} = 
 \intO{ (\Laph \cch^k -\Fc_h^k ) (D_t \cch^k  + \vuh^k \cdot \Gradh \cch^k)}
\\& 
= \intO{ (\Laph \cch^k -\Fc_h^k )  \vuh^k \cdot \Gradh \cch^k} 
-B(\cch^k, D_t \cch^k)  
-  \intO{ \Fc_h^k   D_t \cch^k} 
\\&
= \intO{ (\Laph \cch^k -\Fc_h^k )  \vuh^k \cdot \Gradh \cch^k} 
- \frac12 D_t \normB{\cch^k}^2 - \frac{\TS}{2} \normB{ D_t \cch^k }^2 
\\&
 \quad - \intOB{ D_t F(\cch^k)  +  \frac{\TS}2 \big(1+ 1_{|\cch^k|> 1} +  3 (\cch^{k,*})^2 1_{|\cch^k|\leq 1} \big)  \abs{D_t \cch^k}^2},
\end{aligned}
\end{equation}
where we have used \eqref{dtB} and the following two applications of Taylor's theory 
\begin{align*}
2  \left( \cch^k \pm 1 \right)  D_t \cch^k
& 
 = D_t (\cch^k \pm 1)^2 + \TS (D_t \cch^k)^2 = D_t F(\cch^k) +  \TS (D_t \cch^k)^2,  \quad \mbox{ for } |\cch^k| \geq1, 
\\
  \left( (\cch^k)^3 -  \cch^{k-1} \right)   D_t \cch^k
 &=\frac14 D_t (\cch^k)^4  +\frac{\TS}{2 }  3 (\cch^{k,*})^2 \abs{D_t \cch^k}^2 - \frac12 D_t (\cch^k)^2 + \frac{\TS}2  \abs{D_t \cch^k}^2
\\& =   D_t F(\cch^k)  +    \frac{\TS}2   \big(3 (\cch^{k,*})^2 +1 \big)  (D_t \cch^k)^2, \quad \mbox{ for } |\cch^k| \leq 1.
\end{align*}
Here $\cch^{k,*} \in \co{\cch^{k-1}}{\cch^k}$ is a Taylor remainder term. 

Finally, combining \eqref{es3} and \eqref{es4} together with \eqref{r1}, we complete the proof, i.e.,  
\begin{align*} 
& D_t \intOB{ \frac{1}{2} \vrh^k  \abs{\auhk}^2  + P(\vrh^k)}
+  D_t \left(\intO{  F(\cch^k)}+ \frac12 \normB{\cch^k}^2  \right)
\\&+  \nu   \norm{\Gradh \vuh^k}_{L^2}^2  + \eta \norm{  \Divh \vuh^k}_{L^2}^2
+ \norm{ D_t \cch^k + \vuh^k \cdot \Gradh \cch^k } _{L^2}^2  
%
\\&
=  - \frac{\TS}{2} \intO{ \vrh^{k-1}|D_t \vuh^k|^2  }
-\frac12  \intfaces{ (\vrh^k)^{\rm up} \abs{\us^k \cdot \vn } \abs{\jump{\auhk}}^2   }  
- h^\eps \intfaces{  \avs{\vrh^k} \abs{\jump{\auhk}}^2}
\\ & \quad 
- \frac{\TS}2 \intO{ P''(\xi)|D_t \vrh^k|^2  }
-   \intfaces{ P''(\zeta) \jump{  \vrh^k } ^2 \left(h^\eps  +\frac12 | \us^k \cdot \vn | \right)}
\\&\quad 
 - \frac{\TS}{2} \normB{D_t \cch^k}^2   
 - \intO{  \frac12 \big(1+ 1_{|\cch^k|> 1} +  3 (\cch^{k,*})^2 1_{|\cch^k|\leq 1} \big) \TS \abs{D_t \cch^k}^2}
 \\& = -  D_{\rm num}^k .
\end{align*}
\end{proof}

\paragraph{Uniform bounds.}
From the energy estimates we derive the following uniform bounds. 
\begin{Lemma}\label{Lm_est}
 Let $(\vrh,\vuh, \cch)$ be a solution to the scheme \eqref{scheme} for $\gamma>1$.  
Then, the following estimates hold:
\begin{subequations}\label{ests}
\begin{align}
\label{est1}
\norm{\vrh |\auh|^2}_{L^{\infty}L^1}  \aleq 1, \quad
\norm{\vrh}_{L^{\infty}L^\gamma}  \aleq 1,\quad
 \norm{\vrh\auh}_{L^\infty L^{\frac{2\gamma}{\gamma+1}}}  \aleq 1, \\ 
 \label{est2}
 \norm{\Gradh \vuh}_{L^2 L^2}\aleq 1, \quad   
 \norm{\Divh \vuh}_{L^{2}L^2}\aleq 1, \quad  
 \norm{\vuh}_{L^{2}L^p}   \aleq 1,\\
\label{est3}
   \sup_{t\in(0,T)} \normB{ \cch(t)  }   \aleq 1, \quad 
\norm{ \Fc_h  } _{L^\infty L^2}  \approx \norm{ \cch  } _{L^\infty L^2} \aleq   \norm{ F(\cch)}_{L^\infty L^1} \aleq 1,
 \\ 
 \label{est4}
    \norm{\cch}_{L^\infty L^p} \aleq \sup_{t\in(0,T)} \normB{ \cch  }  +  \norm{ \cch  } _{L^\infty L^2} \aleq 1, 
 \quad  \norm{ D_t \cch + \vuh \cdot \Gradh \cch } _{L^2L^2}  \aleq 1,
\\  \label{est5} 
\norm{\Laph \cch}_{L^2L^2} \aleq  1, \quad 
\norm{D_t \cch}_{L^2L^{3/2}} \aleq 1.
\end{align}
\end{subequations}
where $\normH{\cdot}$ and $\normB{\cdot}$ are defined in \eqref{seminorm}, 
$p \in [ 1,\infty)$ if $d=2$ or $p \in [1,6]$ if $d=3$.
\end{Lemma}
\begin{proof}
Applying the Sobolev-Poincar\'e inequality Lemma~\ref{lm_SP} to the energy estimates stated in Theorem~\ref{thm_sta} we directly obtain the  estimates \eqref{est1}--\eqref{est4}.  
We are left with the proof of \eqref{est5}. 
First, we set $\psi_h = \Laph \cch^k$ in \eqref{schemeC} to get 
\begin{align*}
\norm{\Laph \cch^k}_{L^2}^2 &= 
 \intO{ (D_t \cch^k  + \vuh^k \cdot \Gradh \cch^k) \Laph \cch^k } +  \intO{   \Fc_h^k  \Laph \cch^k   }
 \\& \leq \norm{D_t \cch^k  + \vuh^k \cdot \Gradh \cch^k}_{L^2}^2 + \frac14 \norm{\Laph \cch^k}_{L^2}^2  +  \norm{\Fc_h^k}_{L^2}^2 + \frac14 \norm{\Laph \cch^k}_{L^2}^2. 
\end{align*}
Then we observe the first estimate of \eqref{est5} after 
recalling the bounds $\norm{D_t \cch  + \vuh \cdot \Gradh \cch}_{L^2L^2}\aleq 1$ and $\norm{\Fc_h}_{L^\infty L^2} \aleq 1$.  

Next, due to the uniform bound $\norm{\vuh \cdot \Gradh \cch}_{L^2L^{3/2}} \leq \norm{\vuh}_{L^2L^6} \norm{ \Gradh \cch}_{L^\infty L^2} \aleq 1$ and the second estimate of \eqref{est4} we proves the second estimate of \eqref{est5}, 
which completes the proof. 
\end{proof}

\subsection{Consistency}
\label{sec_C}

Next step towards the convergence of the approximate solutions is the consistency of the numerical scheme. In particular, we require the numerical solution to satisfy the weak formulation of the  continuous problem up to  residual terms  vanishing for $h \to 0.$
\begin{Theorem} \label{Thm_con}
Let $(\vrh, \vuh, \cch)$ be a solution of the approximate problem \eqref{scheme} on the time interval $[0,T]$ with $\TS\approx h$,   $ \gamma  > 4d/(1+3d)$ and the artificial diffusion coefficient $\eps$ satisfies
\begin{equation}\label{eps}
\eps >0 \mbox{ if } \gamma \geq 2 \quad \mbox{ and }\quad   \eps \in(0, 2 \gamma-1 -d/3) \mbox{ if } \gamma \in(4d /(1+3d),2).
\end{equation}
Then
\begin{subequations}
\begin{equation} \label{cP1}
- \intO{ \vrh^0 \phi(0,\cdot) }  =
\int_0^T \intO{ \left[ \vrh \partial_t \phi + \vrh \vuh \cdot \Grad \phi \right]} \dt  + \int_0^T
e_{1,h} (t, \phi) \dt,
\end{equation}
for any $\phi \in C_c^2([0,T) \times \Omega)$ with $\| e_{1,h} (\cdot, \phi ) \|_{L^1(0,T)} \aleq h^\alpha$ for some $ \alpha > 0$;
\begin{equation} \label{cP2}
\begin{split}
- &\intO{ \vrh^0 \widehat{\vuh^0} \cdot \bfphi(0,\cdot) }  =
\int_0^T \intO{ \left[ \vrh \auh \cdot \partial_t \bfphi + \vrh \auh \otimes \vuh  : \Grad \bfphi  + p_h \Div \bfphi \right]} \dt
\\&
 -\nu  \int_0^T \intO{  \Gradh \vuh : \Grad \bfphi}  \dt -\eta \int_0^T \intO{ \Divh \vuh  \Div \bfphi}  \dt 
\\& + \int_0^T \intO{\big( \Fc_h - \Laph \cch  \big) \Gradh \cch     \cdot \bfphi} \dt
 + \int_0^T e_{2,h} (t, \bfphi) \dt
\end{split}
\end{equation}
for any $\bfphi \in C^2_c([0,T) \times {\Omega}; \mathbb{R}^d)$ with $\| e_{2,h} (\cdot, \bfphi ) \|_{L^1(0,T)} \aleq h^\alpha$ for some $ \alpha > 0$;
\begin{equation} \label{cP3}
- \intO{ \cch^0 \psi(0,\cdot) }  =
\int_0^T \intOB{  \cch \partial_t \psi -  \vuh \cdot \Grad \cch \psi    + (\Laph \cch - \Fc_h ) \psi}  \dt  
+ \int_0^T e_{3,h} (t, \psi) \dt,
\end{equation}
for any $\psi \in  C_c^1 ([0,T) \times \Omega)$ with $\| e_{3,h} (\cdot, \psi ) \|_{L^1(0,T)} \aleq h^\alpha$ for some $ \alpha > 0$.
\end{subequations}
\end{Theorem}
\begin{proof}
\textbf{Step 1: proof of \eqref{cP1}.}
Recalling the first estimate of \cite[Theorem 13.2]{FeLMMiSh} we have \eqref{cP1}. 

\textbf{Step 2: proof of \eqref{cP2}.}
Recalling the second estimate of \cite[Theorem 13.2]{FeLMMiSh}, we know there exists a positive constant $\alpha$ such that 
\begin{equation}\label{c1}
\begin{split}
&\intTO{ D_t  (\vrh \auh) \cdot \bfphi_h } - \int_0^T \intfaces{ \bFup(\vrh \auh,\vuh)  \cdot \jump{\widehat{ \bfphi_h} }   } \dt 
\\&
 +\nu  \int_0^T \intO{  \Gradh \vuh : \Gradh \bfphi}  \dt +\eta \int_0^T \intO{ \Divh \vuh  \Divh \bfphi}  \dt 
-  \intTO{ p_h  \Divh \bfphi_h  } 
\\&=
- \intO{ \vrh^0 \widehat{\vuh^0} \cdot \bfphi(0,\cdot) }  -
\int_0^T \intO{ \left[ \vrh \auh \cdot \partial_t \bfphi + \vrh \auh \otimes \vuh  : \Grad \bfphi  + p_h \Div \bfphi \right]} \dt,
\\& \quad 
  +\nu  \int_0^T \intO{  \Gradh \vuh : \Grad \bfphi}  \dt +\eta \int_0^T \intO{ \Divh \vuh  \Div \bfphi}  \dt 
 + c h^\alpha
\end{split}
\end{equation}
 for any $\bfphi \in C^2_c([0,T) \times \Omega; \R^d)$ 
with $\bfphi_h = \Piv \bfphi$, 
where $c$ depends on $\norm{\bfphi}_{C^2}$ and on the initial energy of the problem. 
Comparing the left hand side of \eqref{c1} with the momentum method \eqref{schemeM} we are left to treat the consistency of $\intTO{ (f_h - \Laph \cch) \Gradh \cch \cdot  \bfphi_h}$, which reads
\begin{equation}\label{c2}
\intTO{ (f_h - \Laph \cch) \Gradh \cch \cdot ( \bfphi_h -\bfphi)}
\aleq h^2 \norm{f_h - \Laph \cch}_{L^2L^2} \norm{\Gradh \cch}_{L^2L^2} \norm{\bfphi}_{C^2}
\aleq h^2,
\end{equation}
where we have used \eqref{n4c2}, the uniform bounds  \eqref{est3}, and the first estimate in \eqref{est5}. 
Obviously, combining \eqref{c1} and \eqref{c2} proves \eqref{cP2}, i.e.
\begin{equation*}
\begin{split}
- &\intO{ \vrh^0 \widehat{\vuh^0} \cdot \bfphi(0,\cdot) }  =
\int_0^T \intO{ \left[ \vrh \auh \cdot \partial_t \bfphi + \vrh \auh \otimes \vuh  : \Grad \bfphi  + p_h \Div \bfphi \right]} \dt,
\\&
 -  \nu  \int_0^T \intO{  \Gradh \vuh : \Grad \bfphi}  \dt  - \eta \int_0^T \intO{ \Divh \vuh  \Div \bfphi}  \dt 
+ \int_0^T \intO{ ( \Fc_h  -  \Laph \cch ) \Gradh \cch \cdot \bfphi }
 + c h^\alpha
\end{split}
\end{equation*}
for some $\alpha>0$, where the positive constant $c$ depends on $\norm{\bfphi}_{C^2}$ and the initial energy of the problem.  

\textbf{Step 3: proof of \eqref{cP3}.} 
Let $\psi_h =\Pix \psi$ 
be the test function in \eqref{schemeC} for $\psi \in C_c^1([0,T) \times \Omega)$. Thanks to H\"older's inequality and the uniform bounds \eqref{est3}--\eqref{est5} we first calculate 
\begin{equation}\label{c3}
\begin{aligned}
& \abs{ \intTO{ \big( (D_t \cch + \vu_h \cdot \Gradh \cch) - \Laph \cch +  \Fc_h  \big)  (\psi_h -\psi)} }
\\&
\leq   \norm{(D_t \cch + \vu_h \cdot \Gradh \cch) - \Laph \cch + \Fc_h  }_{L^2L^2} \norm{\psi_h-\psi}_{L^2L^2}
\\& \aleq h \left( \norm{D_t \cch + \vu_h \cdot \Gradh \cch}_{L^2 L^2} 
+ \norm{  \Laph \cch  }_{L^2L^2} + \norm{  \Fc_h  }_{L^\infty L^2} \right)\norm{\psi}_{C^1} 
 \aleq h.
 \end{aligned}
 \end{equation}
 Next, we rewrite the term $\intTO{ D_t \cch \psi }$ as
 \begin{equation}\label{c4}
\intTO{ D_t \cch \psi } = \intTO{ D_t \cch  ( \psi - \psi^{k-1}) }  +  \intTO{D_t \cch \psi^{k-1}} = : I_1 +I_2.
 \end{equation}
For the term $I_1$ we have by Taylor's theory and the estimate \eqref{est5} that
 \begin{align*}
\abs{ I_1 } \leq  \norm{ D_t \cch}_{L^2L^{3/2}} \TS \norm{\psi}_{C^1} \aleq \TS .
\end{align*}
Further, we treat the term $I_2$ in the following way 
 \begin{align*}
I_2 &= \intTO{D_t \cch \psi^{k-1}}  =\sum_{k=1}^{N_T} \TS \intO{D_t \cch \psi^{k-1}}
= \sum_{k=1}^{N_T}  \intO{(\cch^k -\cch^{k-1}) \psi^{k-1}}
\\& 
 = - \sum_{k=1}^{N_T}  \intO{\cch^k (\psi^k - \psi^{k-1}) } +  \intO{\cch^{N_T} \underbrace{ \psi^{N_T}}_{=0}} - \intO{\cch^0 \psi^0}
\\& = -   \intTO{\cch \pd_t \psi }  - \intO{\cch^0 \psi^0}.
\end{align*}
Substituting the above relation together with the estimate of the term $I_1$ into \eqref{c4} implies
 \begin{equation}\label{c5}
\abs{ \intTO{ D_t \cch \psi }  + \intTO{\cch \pd_t \psi }  + \intO{\cch^0 \psi^0} }  \aleq \TS\approx h. 
 \end{equation}
Finally, combining \eqref{c3} with \eqref{c5}  yields
\[
 \intTO{ - \cch \pd_t \psi - \cch^0 \psi^0 + ( \vu_h \cdot \Gradh \cch  - \Laph \cch +  \Fc_h)  \psi }  \aleq h ,
\]
which proves \eqref{cP3}, and completes the proof of Theorem~\ref{Thm_con}. 
\end{proof}

\subsection{Convergence}
In this subsection, we prove the final result, that is the convergence of the numerical solutions resulting from the DG--FE method.

\begin{Theorem}\label{Thm3}
Let $(\vrh,\vuh, \cch )$ be a solution of the DG--FE method \eqref{scheme}, with $\TS  \approx h$, $ \gamma  >  3/2$ for $d=3$ and $ \gamma >8/7$ for $d=2$,  
the artificial diffusion coefficient $\eps$ satisfies \eqref{eps}, 
and the initial data satisfying
\[
\vr_0  \in  L^\gamma(\Omega), \ \vr_0>0,\  \vu_0 \in L^2 (\Omega; \R^d), \ \cc_0 \in W^{1,2}(\Omega) .
\]
\begin{enumerate}
    \item 
Then,  for a suitable subsequence, 
\begin{equation} \label{limit}
\begin{split}
\vrh &\to \vr \ \mbox{weakly-(*) in}\ L^\infty(0,T; L^\gamma(\Omega)), \\ \vuh &\to \vu\ \mbox{weakly in}\ L^2((0,T) \times \Omega; \R^d), \\ \cch &\to \cc \ \mbox{weakly-(*) in}\ L^\infty (0,T; L^2(\Omega)),
\end{split}
\end{equation}
where $\{ \vr , \vu, \cc \}$ is a dissipative weak solution of the Navier--Stokes--Allen--Cahn system \eqref{i1}--\eqref{i2} in the sense of Definition~\ref{DD1}.
 \item 
In addition, suppose that the Navier--Stokes--Allen--Cahn system \eqref{i1}--\eqref{i2} 
with the initial data $(\vr_0, \vu_0, \cc_0)$ admits a strong solution in the class \eqref{cond_strong}.

Then the limit in \eqref{limit} is unconditional (no need of a subsequence) and the limit 
quantity $\{ \vr, \vu, \cc \}$ coincides with the strong solution.

\end{enumerate}
 
\end{Theorem}

\begin{proof}

From the energy estimates \eqref{sta} (see also Lemma \ref{Lm_est}) and the Closed Graph Lemma~\ref{cL1} we deduce that at least for suitable subsequences,
\begin{align*}
\vrh &\to \vr \ \mbox{weakly-(*) in}\ L^\infty(0,T; L^\gamma(\Omega)),\ \vr \geq 0 ,
\\
\vuh, \auh &\to \vu \ \mbox{weakly in}\ L^2((0,T) \times \Omega; \R^d),
\\
 \Gradh \vuh &\to \Grad  \vu \ \mbox{weakly in} \ L^2((0,T) \times \Omega;\R^{d \times d}), \quad  \mbox{where} \ \bu   \in L^2(0,T; W^{1,2}(\Omega; \R^d)), 
 \\
\vrh \vuh, \vrh \auh  &\to \vc{m} \ \  \mbox{weakly-(*) in}\ L^\infty(0,T; L^{\frac{2\gamma}{\gamma + 1}}(\Omega; \R^d)), 
\end{align*}
\begin{equation} \label{esonc}
\begin{split}
\cch &\to \cc \ \mbox{weakly-(*) in}\ L^\infty (0,T; L^2(\Omega)),
  \\
  \Gradh \cch & \to \Grad  \cc \ \mbox{weakly-(*) in}\ L^\infty(0,T; L^2(\Omega;\R^d)), 
 \\
 \Laph \cch & \to \Lap  \cc \ \mbox{weakly in} \ L^2((0,T)\times\Omega), \quad 
 \mbox{where} \ \cc   \in L^\infty(0,T; W^{1,2}(\Omega)) \cap L^2(0,T; W^{2,2}(\Omega)), 
 \\
 \Fc_h &  \to f(\cc)=F'(\cc) \ \mbox{weakly-(*) in}\ L^\infty (0,T; L^2(\Omega)).
\end{split}
\end{equation}

Moreover, by virtue of the same arguments as in \cite[Lemma 7.1]{Karper} (see also \cite[Section 8]{KN2020}), we have $\vm = \vr \bu$, and 
\[ 
\vrh \auh \otimes \vuh  \to \vr \vu \otimes \vu \ \mbox{weakly in}\ L^1 ((0,T) \times \Omega; \R^{d\times d}_{\rm sym} ).
\]

Similarly, combining the estimates on the discrete time derivative $D_t \cc_h$ \eqref{est5} with \eqref{esonc} we obtain 
\[
\cc_h \to \cc \ \mbox{in}\ L^2((0,T) \times \Omega). 
\]

Further, employing the compactness Lemma~\ref{cL2} we find 
\[
\nabla_h \cc_h \to \Grad \cc \ \mbox{in}\ L^2((0,T) \times \Omega; \R^d). 
\]

Finally, it follows from the energy estimates \eqref{sta} and hypothesis \eqref{H3bis} that 
\begin{align*}
P(\vr_h) &\to \Ov{P(\vr)} \ \mbox{weakly-(*) in}\ L^\infty(0,T; \mathcal{M}^+(\Omega)),\\ 
p(\vr_h) &\to \Ov{p(\vr)} \ \mbox{weakly-(*) in}\ L^\infty(0,T; \mathcal{M}^+(\Omega)),
\end{align*}
where 
\[
0 \leq (\Ov{p(\vr)} - p(\vr) ) \mathbb{I} \equiv \mathfrak{R} \aleq  \mathfrak{E} \equiv
\Ov{P(\vr)} - P(\vr), 
\]
see \cite[Section 3.4]{AbbFeiNov} for details. 

Passing to the limit for $h \to 0$ in the consistency formulation \eqref{cP1}--\eqref{cP3} 
and the energy inequality \eqref{sta} we deduce that $(\vr, \vu, \cc)$ is a dissipative weak solution in the sense of Definition \ref{DD1}  and thus, by virtue of the weak-strong uniqueness we can conclude that the limit coincides with the strong solution, provided it exists. 
\end{proof}

\centerline{\textbf{ Acknowledgements}}

This work was supported by the mobility project 8J20FR007 Barrande 2020 of collaboration between France and Czech Republic. The grantor in the Czech Republic is the Ministry of Education, Youth and Sports.  

\appendix
\section{Appendix}
\subsection{Useful equality for the diffusive upwind flux}\label{Sec_Ap_flux}
Here we prove Lemma~\ref{lm_flux}. 
\begin{proof}First, recalling the discrete operators defined in Section~\ref{Sec_nt} we obtain by direct calculation that
\begin{align*}
&\intfacesB{ \bUp(\vrh \auh, \vuh) \cdot \jump{\auh}  -   \Up(\vrh, \vuh) \jump{\frac12|\auh|^2} }
\\& 
= \intfaces{ \vrh^{\up} \Big(\auh^\up \cdot \jump{\auh} - \frac12\jump{|\auh|^2}\Big)  \us \cdot \vn}
\\&
= \intfaces{ \vrh^{\rm in} \jump{\auh}\cdot \left(\auh^{\rm in} - \avs{\auh}\right)  [\vuh \cdot \vn]^+}
\\& \quad + \intfaces{ \vrh^{\rm out} \jump{\auh}\cdot\left(\auh^{\rm out} - \avs{\auh}\right)  [\us \cdot \vn]^-}
\\&
= -\frac12 \intfaces{ \vrh^{\rm in} \abs{ \jump{\auh}}^2  [\us \cdot \vn]^+}
 +\frac12 \intfaces{ \vrh^{\rm out}  \abs{ \jump{\auh}}^2  [\us \cdot \vn]^-}
\\&
= -\frac12 \intfaces{ \vrh^{\up}   \abs{\jump{\auh}}^2  \abs{\us \cdot \vn}}. 
\end{align*}
Next, it is easy to get
\begin{align*}
 - h^\eps \intfaces{\jump{\vrh \auh} \cdot \jump{\auh}} 
 + h^\eps \intfaces{\jump{\vrh } \cdot \jump{\frac12 |\auh|^2}}= - {h^\eps}\intfaces{  \avs{\vrh}  \abs{ \jump{\auh}}^2 }.
\end{align*}
Summing above the above two identities we finish the proof of Lemma~\ref{lm_flux}.
\end{proof}

\subsection{Existence of a numerical solution}\label{Sec_Ap_exi}
Here we prove Lemma~\ref{lem_exist}, that is the existence of a solution and positivity of density for the numerical method \eqref{scheme}. 
We shall show the proof via a topological degree theory, which was reported in Gallou\"{e}t et al.~\cite{GallouetMAC}. 

\begin{Theorem}(\cite[Theorem A.1]{GallouetMAC} Topological degree theory.\label{thm_topo})

Let $M$ and $N$ be two positive integers. Let $  C_1>\eps>0$  and $C_2>0$ be real numbers. Let
\begin{align*}
&V = \{ (r,U) \in R ^M \times R ^N; \ r_i >0 \; \forall \; i=1,\dots, M  \}, \\
&W = \{ (r,U) \in R ^M \times R ^N; | U | \leq C_2  \text{ and }   \eps < r_i < C_1 \; \forall \; i=1, \dots, M    \}.
\end{align*}
Let  $\calF$ be a continuous function mapping $V \times [0,1]$ to $R ^M \times R ^N$ and satisfying:
\begin{enumerate}
\item $ f \in W $ if $ f \in V $  satisfies $ F(f,\zeta)=\mathbf{0}$ for all $ \zeta \in [0,1] $;
\item The equation $\calF(f, 0)=\mathbf{0}$ is a linear system with respect to $f$ and admits a solution in $W$.
\end{enumerate}
Then there exists $ f \in W$ such that $\calF(f,1) =\mathbf{0}$.
\end{Theorem}
Now we are ready to prove Lemma~\ref{lem_exist}. 
\begin{proof}[\bf Proof of Lemma~\ref{lem_exist}.] The idea of the proof is to construct a mapping $\calF$ that satisfies Theorem~\ref{thm_topo}. 
We begin with the definition of the spaces $V$ and $W$  
\begin{align*}
&V =\left\{ ( \vrh^k, \Uh^k) \in \Qh \times \vQh,\ \vrh^k>0  \right\},
\\
&W =\left\{ ( \vrh^k, \Uh^k  ) \in \Qh \times \vQh ,\ \norm{\Uh^k} \leq C_2,  \epsilon < \vrh^k <C_1 \right\},
\end{align*}
where 
$\Uh:=(\vuh, \cch) \in \bVh\times \Xh =: \vQh$, $\vrh>c$ means $\vr_K>c$ for all $K\in \grid$, and the norm $\norm{ U_h}$ is given by  
$\norm{ U_h} \equiv \norm{\vuh}_{L^6} + \norm{ \cch }_{L^6}$.  Obviously, the dimension of the spaces $\Qh$ and $\vQh$ is finite. 

Next, for $\zeta \in [0,1]$ and $U^\star=(\vu^\star, \cc^\star, \mu^\star)$ we define the following mapping 
  \begin{align*}
   \calF : \ V  \times [0,1]\rightarrow  \Qh  \times \vQh,   \quad
        (\vrh^k,\Uh^k, \zeta)\longmapsto ( \vr^\star, U^\star) =\calF(\vrh^k,\Uh^k ,\zeta),
  \end{align*}
where $( \vr^\star, U^\star)$ is defined by:
\begin{subequations}\label{eqA1}
\begin{equation} \label{eqA11}
\intO{ \vr^\star \phi_h } = \intO{ \frac{\vrh^k  -\vrh^{k-1} }{\TS} \phi_h }  - \zeta \intfaces{\Fup ( \vrh^{k}, \vuh^{k} ) \jump{\phi_h} }  ,
\end{equation}
\begin{multline} \label{eqA12}
\intO{  \vu^{\star} \cdot \bfphi_h } =
\intO{ \frac{\vrh^k  \auhk- \vrh^{k-1}  \widehat{\vuh^{k-1} } }{\TS}  \cdot \bfphi_h }
+ \nu \intO{ \Gradh \vuh^k: \Gradh \bfphi_h }
+ \zeta \eta \intO{\Divh \vuh^k \Divh \bfphi_h}
\\
- \zeta \intfaces{ \bFup(\vrh^k \auhk,\vuh^k)
\cdot \jump{\widehat{ \bfphi_h} }   } 
- \zeta \intO{ p(\vrh^{k}) \Divh \bfphi_h } 
- \zeta \intO{ ( \zeta \Fc_h^k - \Laph \cch^k) \Gradh \cch^k \cdot  \bfphi_h},
\end{multline}
\begin{equation}\label{eqA13}
\intO{\cc^\star \psi_h}=  \intO{\frac{ \cch^k-\cch^{k-1}}{\TS}\psi_h}  + \zeta\intO{ \vuh^k \cdot \Gradh \cch^k \psi_h } - \intO{ \big(\Laph \cch^k - \zeta  \Fc_h^k  \big)   \psi_h   } ,
\end{equation}
\end{subequations}
for any $\phi_h \in \Qh$ and $\bfphi_h \times \psi_h  \in \vQh$, where $\bfphi_h=(\phi_{1,h}, \dots, \phi_{d,h})$, 
and the discrete Laplace is defined by the following equality
\[ - \intO{\Laph \cch \psi_h} = (1-\zeta) \intO{\cch \psi_h } + B(\cch,\psi_h) .
\] 
%

It is obvious that $\calF$ is well defined and continuous since the values of $\vr^\star$ and $U^\star=(\vu^\star, \cc^\star)$ can be determined by setting $\phi_h = 1_{K}$,  $K\in \grid$, in \eqref{eqA11}, $\phi_{i,h}=1_\sigma, \sigma \in \faces$ with $\phi_{j,h}=0$ for $j\neq i, \; i,j \in(1,\ldots, d)$ in \eqref{eqA12}, and $\psi_h =1_{P}$,  $P$ being a degree of freedom of the space $\Xh$ in \eqref{eqA13}, respectively. 

With the above definitions, we aim to show that both hypotheses of Theorem~\ref{thm_topo} hold. 

We first aim to prove that Hypothesis~1 of Theorem~\ref{thm_topo} holds. To this end, we suppose $(\vrh^k,\Uh^k ) \in \Qh \times \vQh$ is a solution to $\calF(\vrh^k,\Uh^k , \zeta)= \mathbf{0}$ for any $\zeta \in [0,1]$. Then the system \eqref{eqA1} becomes
\begin{subequations}\label{eqA2}
\begin{equation} \label{eqA21}
 \intO{ \frac{\vrh^k  -\vrh^{k-1} }{\TS} \phi_h }  - \zeta \intfaces{\Fup ( \vrh^{k}, \vuh^{k} ) \jump{\phi_h} } =0,
\end{equation}
\begin{multline} \label{eqA22}
\intO{ \frac{\vrh^k  \auhk- \vrh^{k-1}  \widehat{\vuh^{k-1} } }{\TS}  \cdot \bfphi_h }
+ \nu \intO{  \Gradh \vuh^k: \Gradh \bfphi_h }
+ \zeta \eta \intO{\Divh \vuh^k \Divh \bfphi_h}
\\
- \zeta \intfaces{ \bFup(\vrh^k \auhk,\vuh^k)
\cdot \jump{\widehat{ \bfphi_h} }   } 
- \zeta \intO{ p(\vrh^{k}) \Divh \bfphi_h } 
- \zeta \intO{ (  \zeta  \Fc_h^k - \Laph \cch^k) \Gradh \cch^k \cdot  \bfphi_h}=0,
\end{multline}
\begin{equation}\label{eqA23}
   \intO{\frac{ \cch^k-\cch^{k-1}}{\TS}\psi_h}  + \zeta\intO{ \vuh^k \cdot \Gradh \cch^k \psi_h } - \intO{  \big(\Laph \cch^k -  \zeta  \Fc_h^k  \big)  \psi_h   } =0,
\end{equation}

\end{subequations}
Taking $\phi_h=1$ as a test function in \eqref{eqA21} we immediately obtain
\begin{equation}\label{eqA3}
\norm{\vrh^k}_{L^1}= \intO{\vrh^k} =  \intO{\vrh^{k-1}} \equiv M_0 >0.
\end{equation}
Further, following the proof of the energy stability~\eqref{sta} we know that  
\begin{align*} 
& 
D_t \intOB{ \frac{1}{2} \vrh \abs{\auhk}^2   + \frac{1}{2} \normB{\cch^k}^2+   \zeta  F(\cch^k) + \frac12 (1-\zeta)    |\cch^k|^2+ P(\vrh^k)}+  \nu   \norm{\Gradh \vuh^k}_{L^2}^2 
\\&
 +  \zeta \eta \norm{  \Divh \vuh^k}_{L^2}^2
+ \norm{ D_t \cch^k + \zeta \vuh^k \cdot \Gradh \cch^k } _{L^2}^2   
\leq 0
\end{align*}
Then we may apply the discrete Sobolev's inequality stated in Lemma~\ref{lm_SP} to derive  
\begin{equation}\label{eqA4}
\norm{\Uh^k } \equiv \norm{ \vuh^k}_{L^6} + \norm{ \cch^k }_{L^6} \leq C_2,
\end{equation}
where $C_2>0$ depends on the initial data of the problem.

Next, let $K\in \grid$ be such that $\vr_K^k = \min_{L \in \grid} \vrh^k|_L$. Now setting $\phi_h = 1_K$ and noticing $\jump{\vrh^k}_{\sigma \in \facesK} \geq 0 $ we find 
\begin{align*}
&\frac{ \abs{K} }{\TS \zeta} (\vr_K^k  - \vr^{k-1}_{K})  =- \sum_{\sigma \in \facesK} \intSh{ \Fup (\vrh^k, \vuh^k) }
=  -\sum_{\sigma \in \facesK}  \abs{\sigma} \left( \vrh^{k,\up} \us^k \cdot \vn  - h^\eps \jump{\vrh^k} \right)
\\&
\geq  - \sum_{\sigma \in \facesK} \abs{\sigma}  \vr_K^k \us^k \cdot \vn
+ \sum_{\sigma \in \facesK} \abs{\sigma}  (\vr_K^k - \vrh^{k,\up} ) \us^k \cdot \vn
\\&
= - \abs{K} \vr_K^k  (\Divh \vuh^k)_K
- \sum_{\sigma \in \facesK} \abs{\sigma}  \jump{\vrh^k} [\us^k \cdot \vn]^-
 \geq - \abs{K} \vr_K^k  (\Divh \vuh^k)_K
\\&
\geq - \abs{K} \vr_K^k  \abs{\Divh \vuh^k}_K .
\end{align*}
Thus
$
\vrh^k \geq  \vr_K^k \geq  \frac{\vr^{k-1}_{K} }{1 + \TS \zeta  \abs{(\Divh \vuh^k)_K} }  >0.
$
Consequently, by virtue of \eqref{eqA4}
$ \vrh^k > \epsilon $,
where $\epsilon$ depends only on the data of the problem.
Further, we get from \eqref{eqA3} that
$ \vrh^k \leq \frac{  M_0 }{\min_{K\in \grid} |K|}$,
which indicates the  existence of $C_1>0$ such that
$\vrh^k <C_1$.
Therefore,  Hypothesis 1 of Theorem \ref{thm_topo} is satisfied.

We also need to show that  Hypothesis 2 of Theorem \ref{thm_topo} is satisfied.
Let $\zeta=0$ then the system $\calF(\vrh^k,\Uh^k,0 )=\mathbf{0}$ reads
\begin{subequations}\label{eqA5}
\begin{equation} \label{eqA51}
\vrh^k = \vrh^{k-1},
\end{equation}
\begin{equation}\label{eqA52}
\intO{ \frac{\vrh^k \auhk  - \vrh^{k-1} \widehat{ \vuh^{k-1} } }{\TS}  \cdot \bfphi_h }
+  \nu \intO{ \Gradh \vuh^k :   \Gradh \bfphi_h  }
=0,
\end{equation}
\begin{equation}\label{eqA53}
\begin{split}
 & \intO{ \frac{\cch^k -\cch^{k-1} }{\TS}  \psi_h } = \intO{ \Laph \cch^k    \psi_h  } 
 =  - \intOB{ \cch^k \psi_h  + \Gradh \cch^k  \cdot \Gradh \psi_h   } 
  \\& \qquad   
  - \intfacesB{ \jump{\psi_h} \vn \cdot \avs{\Gradh \cch^k}  +  \jump{\cch^k} \vn \cdot \avs{\Gradh \psi_h}  + \frac{1}{h^{1+\beta}} \jump{\cch^k} \jump{\psi_h}  }.
\end{split}
\end{equation}
\end{subequations}

From \eqref{eqA51} it is obvious $\vrh^k = \vrh^{k-1}>0$.
Substituting \eqref{eqA51} into \eqref{eqA52} we arrive at a linear system on $\vuh^k$ with a symmetric positive definite matrix. Thus \eqref{eqA52} admits a unique solution. 
Further, noticing \eqref{eqA53} is a linear system of $\cch^k$ with a positive definite matrix, we know that it admits a unique solution $\cch^k$. 

Consequently, Hypothesis 2 of Theorem~\ref{thm_topo} is satisfied.

We have shown that both hypotheses of Theorem~\ref{thm_topo} hold. Applying Theorem~\ref{thm_topo} finishes the proof of Lemma~\ref{lem_exist}. 
\end{proof}

\bibliographystyle{plain}

\def\cprime{$'$} \def\ocirc#1{\ifmmode\setbox0=\hbox{$#1$}\dimen0=\ht0
  \advance\dimen0 by1pt\rlap{\hbox to\wd0{\hss\raise\dimen0
  \hbox{\hskip.2em$\scriptscriptstyle\circ$}\hss}}#1\else {\accent"17 #1}\fi}

\end{document}